\documentclass{amsart}
\usepackage{newtxtext, newtxmath}

\usepackage{geometry}
\usepackage{latexsym, bm}
\usepackage{amsmath, amssymb, amsfonts, amsthm, amscd}
\usepackage{cases}
\usepackage{enumerate, paralist}
\usepackage{appendix}

\usepackage[colorlinks=true]{hyperref}
\hypersetup{
	linkcolor=blue,
	anchorcolor=blue,
	citecolor=blue,
	filecolor=blue,
	urlcolor=blue,
	menucolor=blue,
}

\usepackage[msc-links, alphabetic, abbrev, nobysame]{amsrefs}

\usepackage{aliascnt}

\theoremstyle{theorem}
\newtheorem{thm}{Theorem}[section]

\newaliascnt{corollary}{thm}
\newtheorem{cor}[corollary]{Corollary}
\aliascntresetthe{corollary}

\newaliascnt{lemma}{thm}
\newtheorem{lem}[lemma]{Lemma}
\aliascntresetthe{lemma}

\newaliascnt{sublemma}{thm}
\newtheorem{slem}[sublemma]{Sublemma}
\aliascntresetthe{sublemma}

\theoremstyle{definition}
\newaliascnt{definition}{thm}
\newtheorem{defi}[definition]{Definition}
\aliascntresetthe{definition}

\newaliascnt{example}{thm}
\newtheorem{exa}[example]{Example}
\aliascntresetthe{example}

\newaliascnt{remark}{thm}
\newtheorem{rmk}[remark]{Remark}
\aliascntresetthe{remark}

\theoremstyle{theorem}
\newtheorem{thmA}{Theorem}

\allowdisplaybreaks

\numberwithin{equation}{section}

\begin{document}

\title[Regularity]{Regularity properties of some perturbations of non-densely defined operators with applications}

\author{Deliang Chen}
\address{Department of Mathematics, Shanghai Jiao Tong University, Shanghai 200240, People's Republic of China}
\email{chernde@sjtu.edu.cn}
\thanks{Part of this work was done at East China Normal University. The author would like to thank Shigui Ruan, Ping Bi and Dongmei Xiao for their useful discussions and encouragement. The author is grateful to the referee(s) for useful comments and suggestions and particularly pointing out lemma \ref{lem:presentation} and a mistake in theorem \ref{thm:final}, which improved significantly the presentation of the original manuscript.}



\begin{abstract}
This paper is to study some conditions on semigroups, generated by some class of non-densely defined operators in the closure of its domain, in order that certain bounded perturbations preserve some regularity properties of the semigroup such as norm continuity, compactness, differentiability and analyticity. Furthermore, we study the critical and essential growth bound of the semigroup under bounded perturbations. The main results generalize the corresponding results in the case of Hille-Yosida operators. As an illustration, we apply the main results to study the asymptotic behaviors of a class of age-structured population models in $ L^p $ spaces ($ 1 \leq p < \infty $).
\end{abstract}

\maketitle

\setcounter{tocdepth}{2}

\section{Introduction}

The main goal of this paper is to study the preservation of the regularity properties of some class of non-densely defined operators under bounded perturbations. Let $A: D(A) \subset X \rightarrow X$ be a linear operator on some Banach Space $X$ and $B$ a perturbing linear operator. Assume that $A$ has some good regularity properties. Under which conditions can $A + B$ keep the properties of $A$? When $A$ is the generator of a $C_0$ semigroup $T_{A}$, or equivalently $A$ is a Hille--Yosida operator and densely defined, i.e., $\overline{D(A)} = X$, many classes of operator $B$ allow $A+B$ to generate a $C_0$ semigroup $T_{A+B}$, e.g., $B$ is a bounded operator, Desch--Schappacher perturbation, or Miyadera--Voigt perturbation; see \cite{EN00}. If $T_A$ has higher regularity properties, such as immediate/eventual norm continuity, immediate/eventual compactness, immediate/eventual differentiability and analyticity, then one may ask naturally that under what kinds of $B$, these properties can be preserved by $T_{A+B}$. When $B$ is a bounded operator, Nagel and Piazzera \cite{NP98} gave a unified treatment of the problem and found additional conditions assuring the permanence of these regularities. In particular, immediate norm continuity, immediate compactness and analyticity are stable under bounded perturbation \cite{NP98, EN00}. See also \cite{BP01, Pia99} for the case when $B$ is some class of Miyadera--Voigt perturbation.  However, the differentiability is not in this way, see \cite{Ren95} for a counterexample. Pazy \cite{Paz68} gave a condition assuring the permanence of differentiability under bounded perturbation, and Iley \cite{Ile07} pointed out that this condition is also necessary. M{\'a}trai \cite{Mat08a} showed that immediate norm continuity is preserved under Desch--Schappacher perturbation and Miyadera--Voigt perturbation.

However, in many applications, the operator $A$ may be not densely defined or even not a Hille--Yosida operator (see, e.g., \cite{DPS87, PS02, MR07, DMP10}); see also \autoref{model}. Let $A$ be a Hille--Yosida operator \cite{DPS87, ABHN11}, $X_0 := \overline{D(A)}$, $A_0 := A_{X_0}$, where $A_{X_0}$ denotes the part of $A$ in $X_0$, i.e.,
\[
A_{X_0}x = Ax,~~ x \in D(A_{X_0}) := \{x \in D(A): Ax \in X_0\}.
\]
It is well known that $A_0$ generates a $C_0$ semigroup $T_{A_0}$ in $X_0$, and $A+B$ is also a Hille--Yosida operator if $B \in \mathcal{L}(X_0, X)$ \cite{KH89}. A natural question may be asked: Are the regularities of $T_{(A + B)_0}$ the same as $T_{A_0}$? B{\'a}tkai, Maniar and Rhandi \cite{BMR02} dealt with this problem by using extrapolation theory and obtained similar results as in \cite{NP98}.

Magal and Ruan \cite{MR07} studied a more general class of non-densely defined operators which in this paper are called \emph{MR operators} and \emph{quasi Hille--Yosida operators} (see \autoref{def:MR} and \autoref{def:pHY}). These operators turn out to be important for the study of certain abstract Cauchy problems, such as age-structured population models, parabolic differential equations and delay equations (see, e.g., \cite{PS02, MR07, DMP10, MR18, Che18g}). Let $A$ be an MR operator (resp. quasi Hille--Yosida operator) and $T_{A_0}$ the $C_0$ semigroup generated by $A_0$ in $X_0$. It was shown in \cite{MR07, Thi08} that $A$ is stable under the bounded perturbation, that is $A+B$ is still an MR operator (resp. quasi Hille--Yosida operator) for all $B \in \mathcal{L}(X_0, X)$. We are interested in that under which conditions $T_{(A+B)_0}$, generated by $(A+B)_0$ in $X_0$, can preserve the regularity properties of $T_{A_0}$. The first part of the paper addresses the problem and obtains analogous and generalized results as in \cite{NP98, BMR02}, i.e., \emph{norm continuity}, \emph{compactness}, \emph{differentiability} and \emph{analyticity} (see \autoref{regper}). We use the method developed in \cite{NP98, BMR02} and the integrated semigroups theory.

The second part of the paper is to study the stability of the \emph{critical growth bound} and \emph{essential growth bound} (see \autoref{def:crit} and \autoref{def:ess}) of a $C_0$ semigroup generated by the part of an MR operator in the closure of its domain under certain bounded perturbations.
Critical spectrum was introduced independently by Nagel and Poland \cite{NP00} and Blake \cite{Bla01}. The authors used this notion to obtain the partial spectral mapping theorem, which could characterize the stability of $C_0$ semigroups very well, see \cite{NP00} for details. Brendle, Nagel and Poland \cite{BNP00} studied the stability of critical growth bound of a $C_0$ semigroup under some class of Miyadera--Voigt perturbation. In particular, they obtained, under appropriate assumptions, a partial spectrum mapping theorem for the perturbed semigroup. Similar result was also obtained by Boulite, Hadd and Maniar \cite{BHM05} for the Hille--Yosida operators. The stability of the essential growth bound were considered by many authors, e.g., Voigt \cite{Voi80} and Andreu, Mart\'{i}nez and Maz\'{o}n \cite{AMM91} (the generators of $C_0$ semigroups), Thieme \cite{Thi97} (Hille--Yosida operators), Ducrot, Liu and Magal \cite{DLM08} (quasi Hille--Yosida operators). Such results have wide applications, e.g., to study the stability of equilibriums, the existence of center manifolds and Hopf bifurcation, see \cite{EN00, BNP00, ABHN11, MR09, MR09a}. See also \cite{Sbi07, Bre01} for an approach based on the resolvent characterization in Hilbert spaces and \cite[Section 2 and 3]{MS16} for some new partial spectral mapping theorems in an abstract framework. We consider the perturbation of the critical and essential growth bound in the case of MR operators and give a unified treatment of the two problems (see \autoref{criess}). Our proof is close to \cite{BNP00}.

Some simple version of our main results in \autoref{regper} and \autoref{criess} may be summarized below; see those sections for more detailed results and see \autoref{pre} for definitions and notations.
For a linear operator $ A: D(A) \subset X \to X $, let $X_0 = \overline{D(A)}$, $A_0 := A_{X_0}$. $ T_{A_0} $ denotes the $ C_0 $ semigroup (if it exists) generated by $ A_0 $.

\begin{thmA}[norm continuity and compactness]
	Let $A$ be an MR operator (see \autoref{def:MR}), $L \in \mathcal{L}(X_0, X)$.
	\begin{enumerate}[(a)]
		\item Assume $LT_{A_0}$ is norm continuous on $(0, \infty)$. Then, $T_{A_0}$ is eventually (resp. immediately) norm continuous if and only if $T_{(A+L)_0}$ is eventually (resp. immediately) norm continuous.
		\item Assume $LT_{A_0}$ is norm continuous and compact on $(0, \infty)$. Then, $T_{A_0}$ is eventually (resp. immediately) compact if and only if $T_{(A+L)_0}$ is eventually (resp. immediately) compact.
		\item Suppose $A$ is a quasi Hille--Yosida operator (see \autoref{def:pHY}) and $LT_{A_0}$ is compact on $(0, \infty)$. Then, $T_{A_0}$ is eventually norm continuous (resp. eventually compact) if and only if $T_{(A+L)_0}$ is eventually norm continuous (resp. eventually compact); the result also holds when ``eventually'' is replaced by ``immediately''.
	\end{enumerate}
\end{thmA}

The following result can be proved very simply if one uses the corresponding characterization of the resolvent.
\begin{thmA}[differentiability and analyticity]
	\begin{enumerate}[(a)]
		\item Let $A$ be a Hille--Yosida operator. Then, $T_{(A+L)_0}$ is eventually differentiable  for all $L \in \mathcal{L}(X_0, X)$ if and only if $A_0$ satisfies the \emph{Pazy-Iley condition} (i.e., the condition of \autoref{thm:PI} in (b) or (c)).
		\item If $A$ is a $p$-quasi Hille--Yosida operator (see \autoref{def:pHY}) and $A_0$ is a Crandall--Pazy operator satisfying
		\[
		\|R(iy, A_0)\|_{\mathcal{L}(X_0)} = O(|y|^{-\beta}),~ |y| \rightarrow \infty, ~ \beta > 1- \frac{1}{p},
		\]
		then $(A+L)_0$ is still a Crandall--Pazy operator (see, e.g., \cite{Ile07, CP69}) for any $L \in \mathcal{L}(X_0, X)$.
		\item Let $A$ be an MR operator (see \autoref{def:MR}). Then, $T_{A_0}$ is analytic if and only if for any $L \in \mathcal{L}(X_0,X)$, $T_{(A+L)_0}$ is analytic. In particular, if $A$ is an almost sectorial operator (see \autoref{def:almost}), so is $A+L$ for any $L \in \mathcal{L}(X_0,X)$.
	\end{enumerate}
\end{thmA}

Let $ \omega_{\mathrm{crit}}(T_{A_0}) $ and $ \omega_{\mathrm{ess}}(T_{A_0}) $ denote the \emph{critical growth bound} and the \emph{essential growth bound} of $ T_{A_0} $ respectively (see \autoref{def:crit} and \autoref{def:ess}).
\begin{thmA}\label{thm:C}
	\begin{enumerate}[(a)]
		\item Let $A$ be an MR operator (see \autoref{def:MR}) and $L \in \mathcal{L}(X_0,X)$.
		If $LT_{A_0}$ is norm continuous on $(0,\infty)$, then $\omega_{\mathrm{crit}}(T_{(A+L)_0}) = \omega_{\mathrm{crit}}(T_{A_0})$.
		If $LT_{A_0}$ is norm continuous and compact on $(0,\infty)$, then $\omega_{\mathrm{ess}}(T_{(A+L)_0}) = \omega_{\mathrm{ess}}(T_{A_0})$.

		\item If $A$ is a quasi Hille--Yosida operator (see \autoref{def:pHY}) and $LT_{A_0}$ is compact on $(0,\infty)$, then $\omega_{\mathrm{ess}}(T_{(A+L)_0}) = \omega_{\mathrm{ess}}(T_{A_0})$, $\omega_{\mathrm{crit}}(T_{(A+L)_0}) = \omega_{\mathrm{crit}}(T_{A_0})$.
	\end{enumerate}
\end{thmA}

\begin{rmk}
	\begin{enumerate}[(a)]
		\item The above results are known at least for the Hille--Yosida operators; see, e.g., \cite{BMR02, BHM05, Thi97}.

		\item \autoref{thm:C} (b) is basically due to \cite{DLM08} but the result is strengthened as $\omega_{\mathrm{ess}}(T_{(A+L)_0}) = \omega_{\mathrm{ess}}(T_{A_0})$; also our proof in some sense simplifies \cite{DLM08}.
	\end{enumerate}
\end{rmk}

As an illustration, in the third part of this paper (see \autoref{model}), we apply the main results in \autoref{regper} and \autoref{criess} to study the asymptotic behaviors of a class of age-structured population models in $ L^p $ ($ 1 \leq p < \infty $). This problem was investigated extensively by many authors, see, e.g., \cite{Web85, Web08, Thi91, Thi98, Rha98, BHM05, MR09a} etc in the $ L^1 $ case. It seems that the $ L^p $ ($ p > 1 $) case was first investigated in \cite{MR07}. As a motivation, in control theory (and approximation theory), age-structured population models can be considered as boundary control systems and in this case the state space is usually taken as $ L^p $ ($ p > 1 $); see, e.g., \cite{CZ95}. Basically, in order to give the asymptotic behaviors of the age-structured population model \eqref{equ:age} in $ L^p $ ($ p > 1 $), the results given in \autoref{regper} and \autoref{criess} are necessary. Concrete examples are given in \autoref{exa:model} to verify different cases in the main result (\autoref{thm:age}).

The paper is organized as follows. In section 2, we recall some definitions and results about integrated semigroups, some classes of non-densely defined operators (with a detailed summary of their basic properties), critical spectrum and essential spectrum. In section 3, we consider the regularities of $S_A \diamond V$ (see section 2 for the symbol's meaning). In section 4, we deal with the regularity properties of the perturbed semigroups generated by the part of MR operators in their closure domains. In section 5, we study the perturbation of critical and essential growth bound. Section 6 contains an application of our results to a class of age-structured population models in $ L^p $. In the final section, we give a relatively bounded perturbation for MR operators associated with their perturbed regularities, and some comments that how the results could be applied to (nonlinear) differential equations.

\section{Preliminaries}\label{pre}

In this section, we recall some definitions and results about what we need in the following, such as integrated semigroups, some classes of non-densely defined operators, the critical spectrum and essential spectrum of $ C_0 $ semigroups.

\subsection{Integrated semigroup}
The integrated semigroup was introduced by W. Arendt \cite{Are87a}. The systematic treatment based on techniques of Laplace transforms was given in \cite{ABHN11}.

Let $X$ and $Z$ be Banach spaces. Denote by $\mathcal{L}(X, Z)$ the space of all bounded linear operators from $X$ into $Z$ and by $\mathcal{L}(X)$ the $\mathcal{L}(X, X)$. Let $A:~ D(A) \subset X \rightarrow X$ be a linear operator and assume $\rho(A) \neq \emptyset$. If $Z \hookrightarrow X$ (i.e., $Z \subset X$ and the imbedding is continuous), $A_Z$ denotes the part of $A$ in $Z$, i.e.,
\[
A_{Z}x = Ax,~~ x \in D(A_{Z}) := \{x \in D(A) \cap Z: Ax \in Z\}.
\]
By the closed graph theorem, $A_Z$ is a closed operator in $Z$ since $A$ is closed. Here is the relationship between $A$ and $A_Z$, see \cite[Proposition B.8, Lemma 3.10.2]{ABHN11}.
\begin{lem}\label{partop}
	Let $Z \hookrightarrow X$.
	\begin{enumerate}[(a)]
		\item If $\mu \in \rho(A)$ and $R(\mu, A)Z \subset Z$, then $D(A_Z) = R(\mu, A)Z$, and $R(\lambda, A)|_Z = R(\lambda, A_Z), ~ \forall \lambda \in \rho(A)$.
		\item If $D(A) \subset Z$, then $R(\lambda, A)Z \subset Z, ~\forall \lambda \in \rho(A)$ and $\rho(A) = \rho(A_Z)$.
	\end{enumerate}
\end{lem}

\begin{defi}[integrated semigroup \cite{ABHN11}]
	Let $A$ be an operator on a Banach space $X$. We call $A$ the generator of (once non-degenerate) integrated semigroup if there exist $\omega \geq 0$ and a strongly continuous function $S: [0, \infty) \rightarrow \mathcal{L}(X)$ satisfying $\|\int_0^t S(s)~\mathrm{d}s \| \leq Me^{\omega t} (t \geq 0)$ for some constant $ M > 0 $ such that $(\omega, \infty) \subset \rho(A)$ and
	\begin{equation}\label{equ:int}
		R(\lambda, A)x =  \lambda \int_0^\infty e^{-\lambda s}S(s)x~\mathrm{d}s, ~\forall \lambda > \omega, ~ \forall x \in X.
	\end{equation}
	In this case, $S$ is called the (once non-degenerate) \emph{integrated semigroup} generated by $A$.
\end{defi}

The equivalent definition is the following, see \cite[Proposition 3.2.4]{ABHN11}.
\begin{lem}\label{def:equ}
	Let $S: [0, \infty) \rightarrow \mathcal{L}(X)$ be a strongly continuous function satisfying $\|\int_0^t S(s)~\mathrm{d}s \| \leq Me^{\omega t} (t \geq 0)$, for some $M,~ \omega\geq 0$. Then, the following assertions are equivalent.
	\begin{enumerate}[(a)]
		\item There exists an operator $A$ such that $(\omega, \infty) \subset \rho(A)$, and \eqref{equ:int} holds.
		\item For $s, ~t \geq 0$,
		\begin{equation}\label{equ:int:11}
			S(t)S(s) = \int_0^t(S(r+s)-S(r))~\mathrm{d}r,
		\end{equation}
		and $S(t)x = 0$ for all $t \geq 0$ implies $x=0$.
	\end{enumerate}
\end{lem}

For some basic properties of integrated semigroups, see \cite[Section 3.2]{ABHN11}. From now on, we set
\[
X_0 := \overline{D(A)}, ~A_0 := A_{X_0}.
\]
Note that, by \autoref{partop}, $\rho(A) = \rho(A_0)$.
\begin{lem}\label{partInt}
	If $A_0$ generates a $C_0$ semigroup $T_{A_0}$ in $X_0$, then $A$ generates an integrated semigroup $S_A$ in $X$. Furthermore, the following hold.
	\begin{enumerate}[(a)]
		\item $T_{A_0}(t)x = \dfrac{d}{dt}S_A(t)x,~\forall x \in X_0$;
		\item $S_A$ can be represented as the following,
		\begin{align}
		S_A(t)x & = (\mu - A_0) \int_0^tT_{A_0}(s)R(\mu, A)x~\mathrm{d}s  \notag \\
		& = \mu \int_0^tT_{A_0}(s)R(\mu, A)x~\mathrm{d}s - T_{A_0}(t)R(\mu, A)x + R(\mu, A)x, ~\forall x \in X,\label{equ:biaodashi}
		\end{align}
		for all $\mu \in \rho(A)$.
		\item $S_A$ is exponentially bounded, i.e., $\|S_A(t)\| \leq M e^{\omega t}, ~ \forall t \geq 0$, for some constants $M \geq 0$ and $\omega \in \mathbb{R}$;
		\item $T_{A_0}(t)S_{A}(s) = S_{A}(t+s) - S_A(t)$, for all $t,s \geq 0$;
		\item $T_{A_0}(\cdot)x$ is continuously differentiable on $[t_0,\infty)$ for any $x \in D(A)$ if and only if $S_A(\cdot)x$ is continuously differentiable on $[t_0,\infty)$ for any $x \in X$, where $t_0 \geq 0$.
	\end{enumerate}
\end{lem}
\begin{proof}
	It suffices to show (b), but this directly follows from the definition of integrated semigroup. Others follow from (b) and \autoref{def:equ}.
\end{proof}

\subsection{Non-densely defined operator}\label{nddo}
Here, we discuss a general class of non-densely defined operators (i.e., $\overline{D(A)} \neq X$), developed by Magal and Ruan \cite{MR07}. Since the importance of the operators and for the sake of our reference, we call them MR operator and $p$-quasi Hille--Yosida operator below.
\begin{defi}[MR operator \cite{MR07}]\label{def:MR}
	We call a closed linear operator $A$ an \textbf{MR operator}, if the following two conditions hold.
	\begin{enumerate}[(a)]
		\item $A_0$ generates a $C_0$ semigroup $T_{A_0}$ in $X_0$, i.e., $A_0$ is a Hille--Yosida operator in $X_0$, and $\overline{D(A_0)} = X_0$. Then, $A$ generates an integrated semigroup $S_A$ in $X$.
		\item There exists an increasing function $\delta: ~[0,\infty) \rightarrow [0,\infty)$ such that $\delta(t) \rightarrow 0$, as $t \rightarrow 0$, and for any $f \in C^1(0,\tau; X), ~\tau > 0$,
		\begin{equation}\label{equ:guji}
			\left\|\frac{d}{dt} \int_0^tS_A(t-s)f(s)~\mathrm{d}s\right\| \leq \delta(t)\sup\limits_{s\in [0,t]}\|f(s)\|, ~\forall t \in [0,\tau].
		\end{equation}
	\end{enumerate}
\end{defi}
In the following, we set
\[
(S_A * f)(t) := \int_0^tS_A(t-s)f(s)~\mathrm{d}s, ~~(S_A \diamond f)(t) := \frac{d}{dt}(S_A * f)(t).
\]

We collect some basic properties about MR operators basically due to \cite{MR07, MR09}; we give a direct proof for the sake of readers.
\begin{lem}\label{lem:MR}
	Let $A$ be an MR operator.
	\begin{enumerate}[(a)]
		\item \label{MRaa} $S_A$ is norm continuous on $[0,\infty)$ and $\|S_A(t)\| \leq \delta(t)$.
		\item $\|R(\lambda, A)\| \rightarrow 0$, as $\lambda \rightarrow \infty, \lambda \in \mathbb{R}$; particularly $ S_A(t)x = \lim_{\mu \to \infty}\int_0^tT_{A_0}(s)\mu R(\mu, A)x~\mathrm{d}s $.
		\item For $\forall f \in C(0,\tau; X)$, $(S_A*f)(t) \in D(A)$, $ t \mapsto (S_A*f)(t) $ is differentiable, $(S_A \diamond f)(t)\in X_0$, $t \mapsto (S_A \diamond f)(t)$ is continuous and \eqref{equ:guji} also holds. Furthermore, the following equations hold.
		\begin{align}
			(S_A \diamond f)(t) & = \lim_{\mu \rightarrow \infty} \int_0^tT_{A_0}(t-s) \mu R(\mu, A)f(s)~\mathrm{d}s \label{equ:sf1}\\
			& = T_{A_0}(t-s)(S_A \diamond f)(s) + (S_A \diamond f(s + \cdot))(t-s) \label{equ:sf2}\\
			& = \lim_{h \rightarrow 0} \frac{1}{h} \int_0^tT_{A_0}(t-s)S(h)f(s)~\mathrm{d}s. \label{equ:sf3}
		\end{align}
	\end{enumerate}
\end{lem}
\begin{proof}
	(a) To show $\|S_A(t)\| \leq \delta(t)$, it suffices to take $ f(s) = x $ in \eqref{equ:guji}. Then, by \autoref{partInt} (d), we know $S_A$ is norm continuous on $[0,\infty)$.

	(b) Note that for $ \lambda > 0 $ and $ \varepsilon = \lambda^{-1/2} $, we have
	\begin{align*}
		\|R(\lambda, A)\| & = \| \lambda \int_{0}^{\infty} e^{-\lambda s}S_{A}(s) ~\mathrm{d}s \| \\
		& \leq \lambda \| \int_{0}^{\varepsilon} e^{-\lambda s}S_{A}(s) ~\mathrm{d}s \| + \lambda \| \int_{\varepsilon}^{\infty} e^{-\lambda s}S_{A}(s) ~\mathrm{d}s \| \\
		& \leq \delta(\varepsilon) + \lambda \| \int_{\varepsilon}^{\infty} e^{-\lambda s} \int_{0}^{s} S_{A}(r) ~\mathrm{d}r ~\mathrm{d} s \| + \lambda \| e^{-\lambda\varepsilon} \int_{0}^{\varepsilon} S_{A}(r) ~\mathrm{d}r \|\\
		& \quad \to 0, ~\text{as}~ \lambda \to \infty.
	\end{align*}
	The representation of $ S_{A} $ now immediately follows from \eqref{equ:biaodashi}.

	(c) For the first statement, see \cite{MR07, Thi08}. We only consider the equations \eqref{equ:sf1} \eqref{equ:sf2} \eqref{equ:sf3}. Since
	\[
	R(\mu, A_0) (S_A \diamond f)(t) = (S_A \diamond R(\mu, A) f)(t) = \int_{0}^{t} T_{A_0} (t - s) R(\mu, A) f(s) ~\mathrm{d} s,
	\]
	we have \eqref{equ:sf1} holds; here for $ x \in X_0 $, $ \mu R(\mu, A_0)x \to x $ as $ \mu \to \infty $. \eqref{equ:sf2} follows from \eqref{equ:sf1} and the property of the convolution operation. Let us consider \eqref{equ:sf3}.
	\begin{align*}
		(S_A \diamond f)(t) & = \lim_{h \rightarrow 0} \frac{1}{h} ((S_A*f)(t+h) - (S_A*f)(t)) \\
		& = \lim_{h \rightarrow 0} \frac{1}{h} [(S_A(h+\cdot)*f)(t) - (S_A*f)(t) + (S_A*f(t+\cdot))(h)] \\
		& = \lim_{h \rightarrow 0} \frac{1}{h} [(S_A(h+\cdot) - S_A(\cdot))*f](t) \quad \text{(by (a))} \\
		& = \lim_{h \rightarrow 0} \frac{1}{h} (T_{A_0}(\cdot)S(h)*f)(t). \quad \text{(by \autoref{partInt} (d))}
	\end{align*}
	This completes the proof.
\end{proof}

Equation \eqref{equ:sf2} will be frequently used in our proof.
Now we turn to other important class of MR operators.
\begin{defi}[$p$-quasi Hille--Yosida operator \cite{MR07}]\label{def:pHY}
	We call a closed linear operator $A$ a \textbf{$p$-quasi Hille--Yosida operator} ($p \geq 1$), if the following two conditions hold.
	\begin{enumerate}[(a)]
		\item $A_0$ generates a $C_0$ semigroup $T_{A_0}$ in $X_0$. Then, $A$ generates an integrated semigroup $S_A$ in $X$.
		\item There exist $\widehat{M} \geq 0, ~\widehat{\omega} \in \mathbb{R}$ such that for any $f \in C^1(0,\tau; X)$, $\tau > 0$
		\begin{equation}\label{equ:pHY}
			\left\| (S_A \diamond f)(t) \right\| \leq \widehat{M} \left\|e^{\widehat{\omega}(t-\cdot)}f(\cdot) \right\|_{L^p(0,t;X)}, ~\forall t \in [0,\tau].
		\end{equation}
	\end{enumerate}
\end{defi}
For $p$-quasi Hille--Yosida operator, if we don't emphasize the $p$, we also call it \emph{quasi Hille--Yosida operator}.
\begin{lem}\label{lem:pHY}Let $A$ be a $p$-quasi Hille--Yosida operator.
	\begin{enumerate}[(a)]
		\item 1-quasi Hille--Yosida operators are Hille--Yosida operators.
		\item For all $f \in L^p(0,\tau;X)$, $(S_A*f)(t) \in D(A)$, $ t \mapsto (S_A*f)(t) $ is differentiable, $(S_A \diamond f)(t)\in X_0$, $t \mapsto (S_A \diamond f)(t)$ is continuous and \eqref{equ:pHY} also holds.
		\item There is $M \geq 0$ such that
		\[
		\|R(\lambda, A)\| \leq \frac{M}{(\lambda-\widehat{\omega})^{\frac{1}{p}}}, ~~\forall \lambda > \widehat{\omega}.
		\]
	\end{enumerate}
\end{lem}
\begin{proof}
	(a) See \cite[Section 4]{Thi08} for more general results.

	(b) See \cite{MR07, Thi08}.

	(c) This is a corollary of \cite[Theorem 4.7 (iii) and Remark 4.8]{MR07}. Indeed, by
	\[
	\|x^*\circ[\lambda-(A-\widehat{\omega})]^{-n}\|_{X^*} \leq \frac{1}{(n-1)!} \int_0^\infty s^{n-1}e^{-\lambda s}\chi_{x^*}(s)~\mathrm{d}s , ~~\lambda >0,
	\]
	where $x^* \in (X_0)^*$, $\sup\limits_{x^* \in (X_0)^*, |x^*| \leq 1}\|\chi_{x^*}\|_{L^p(0,\infty;\mathbb{R})} < \infty$, $n \in \mathbb{N}$, and by the H\"{o}lder inequality for the case $n=1$, we obtain the result. The proof is complete.
\end{proof}

\begin{rmk}\label{rmk:cc}
	A beautiful characterization of $ A $ which is an MR operator or a $ p $-quasi Hille--Yosida operator by using the regularity of the integrated semigroup $ S_A $ was given by Thieme \cite{Thi08}. We state here for the convenience of readers. Suppose for the operator $ A: D(A) \subset X \to X $, $A_0$ generates a $C_0$ semigroup $T_{A_0}$ in $X_0$. Let $ S_A $ be the integrated semigroup generated by $ S_A $.
	\begin{enumerate}[(a)]
		\item $ A $ is an MR operator if and only if $ S_A $ is of \emph{bounded  semi-variation} on $ [0, b] $ for some $ b > 0 $ with its semi-variation $ V_{\infty} (S_A; 0, b) \to 0 $ as $ b \to 0 $. Here we mention that if $ T_{A_0} $ is norm continuous on $ (0, b) $, then $ V_{\infty} (S_A; 0, b) \to 0 $ as $ b \to 0 $. Note also that $ S_A $ is of bounded  semi-variation on $ [0, b] $ if and only if $ x^*S_A $ is of \emph{bounded variation} for all $ x^* \in (X_0)^* $ (see, e.g., \cite[Theorem 2.12]{Mon15}).
		\item $ A $ is a $ p $-quasi Hille--Yosida operator if and only if $ S_A $ is of \emph{bounded semi-$ p $-variation} on $ [0, b] $ for some $ b > 0 $, and if and only if $ x^*S_A $ is \emph{bounded $ p' $-variation} on $ [0, b] $ for all $ x^* \in (X_0)^* $ where $ 1/{p'} + 1/p = 1 $. We notice that for the semi-$ p $-variation $ V_{p} (S_A; 0, b) $, one has $ V_{\infty} (S_A; 0, b) \leq b^{1/p} V_{p} (S_A; 0, b) $.
	\end{enumerate}
	For the notions of (semi-) ($ p $-) variation, see \cite{Thi08, Mon15} for details. Additional remark should be made: for a function $ f: [a,b] \to X $ with $ X $ having Radon-Nikodym property, then $ f $ is of bounded $ p $-variation if and only if $ f \in W^{1,p}((a,b), X) $.
	Also note that for MR operator $ A $, $ (S_A \diamond f)(t) $ can be written in the convolution form by using the Stieltjes-type integral, i.e., (see \cite[Theorem 3.2]{Thi08})
	\[
	(S_A \diamond f)(t) = \int_{0}^{t} \mathrm{d}S_{A}(s) f(s - t).
	\]
\end{rmk}

Quasi Hille--Yosida operators are related to almost sectorial operators.
\begin{defi}[$\frac{1}{p}$-almost sectorial operator \cite{PS02, DMP10}]\label{def:almost}
	We call a closed linear operator $A$ a \textbf{$\frac{1}{p}$-almost sectorial operator} ($p \geq 1$), if there exist $\widehat{M} > 0$, $\widehat{\omega} \in \mathbb{R}$, $\theta \in (\frac{\pi}{2}, \pi)$ such that
	\[
	\|R(\lambda, A)\| \leq \frac{\widehat{M}}{|\lambda - \widehat{\omega}|^{p}},
	\]
	for all $\lambda \in \{\lambda \in \mathbb{C}:~\lambda \neq \widehat{\omega}, ~|\arg(\lambda - \widehat{\omega})| < \theta\}$.
\end{defi}
\begin{lem}\label{lem:almost}
	\begin{enumerate}[(a)]
		\item If $A$ is a $\frac{1}{p}$-almost sectorial operator, then $A_0$ is a generator of an analytic $C_0$ semigroup, and $A$ is a $\widehat{p}$-quasi Hille--Yosida operator for any $\widehat{p} > p$.
		\item If $A$ is a $p$-quasi Hille--Yosida operator and $A_0$ is a generator of an analytic $C_0$ semigroup, then $A$ is a $\frac{1}{p}$-almost sectorial operator.
		\item Let $A_0$ be a generator of an analytic $C_0$ semigroup. Then, $A$ is a Hille--Yosida operator if and only if $A$ is $1$-almost sectorial operator (i.e., the generator of a holomorphic semigroup, see \cite[Definition 3.7.1]{ABHN11}).
	\end{enumerate}
\end{lem}
\begin{proof}
	(a) That $A_0$ is a generator of an analytic $C_0$ semigroup was shown in \cite[Proposition 3.12, Theorem 3.13]{PS02}, and that $A$ is a $\widehat{p}$-quasi Hille--Yosida operator for any $\widehat{p} > p$ was proved in \cite[Theorem 3.11]{DMP10}.

	(b) Combine \cite[Proposition 3.3]{DMP10} and \autoref{lem:pHY} (c).

	(c) See \cite[Theorem 3.7.11 and Example 3.5.9 c)]{ABHN11}.
\end{proof}
Because of the above lemma, we may interpret almost sectorial operators as an analytic version of quasi Hille--Yosida operators.
We refer to \cite{MR07, MR09, MR09a} for more results and examples on MR operators (quasi Hille--Yosida operators), particularly in applications to abstract Cauchy problems, and to \cite{PS02, DMP10, CDN08} on almost sectorial operators; see also \autoref{model} and \autoref{comments}.

\begin{rmk}
	For a generator $ A $ of an integrated semigroup $ S_A $, unlike the case of Hille--Yosida operator (where $ A $ is a Hille--Yosida operator if and only if $ S_A $ is locally Lipschitz, and if and only if $ S_A $ is of locally bounded semi-$ 1 $-variation) and $\frac{1}{p}$-almost sectorial operator, even $ S_A $ is of bounded semi-$ p $-variation ($ p > 1 $) or semi-variation, in general, $ A_{\overline{D(A)}} $ might not generate a $ C_0 $ semigroup; see, e.g., \cite[Example 5.2 and 5.3]{Thi08} where the operators given in that paper are densely defined.
\end{rmk}

Assume that $A$ is an MR operator. Denote by $C_s([0,\infty), \mathcal{L}(X_0, X))$ the space of strongly continuous map from $[0,\infty)$ into $\mathcal{L}(X_0, X)$). Set
\[
(S_A \diamond V)(t)x := (S_A \diamond V(\cdot)x)(t), ~~\forall x \in X_0,
\]
where $V \in C_s([0,\infty), \mathcal{L}(X_0, X))$. Note that for any $t \geq 0$, $(S_A \diamond V)(t) \in \mathcal{L}(X_0, X)$, and $t \mapsto (S_A \diamond V)(t)$ is norm continuous at zero (see \autoref{lem:MR} \eqref{MRaa}). Set
\[
\mathcal{B}(W) := S_A \diamond LW,~~ W \in C_s([0,\infty), \mathcal{L}(X_0, X)),
\]
where $L \in\mathcal{L}(X_0, X)$.
\begin{lem}\label{lem:fix}
	Suppose that $A$ is an MR operator. Then, for any $V \in C_s([0,\infty), \mathcal{L}(X_0, X))$, the equation
	\begin{equation}\label{equ:fix}
		W = V + \mathcal{B}(W)
	\end{equation}
	has a unique solution $W \in C_s([0,\infty), \mathcal{L}(X_0, X))$. Furthermore,
	\[
	W = \sum\limits_{n=0}^{\infty}\mathcal{B}^n(V),
	\]
	where the series is uniformly convergent on any finite interval in the uniform operator topology.
\end{lem}
\begin{proof}
	The result was already stated in the proof of \cite[Theroem 3.1]{MR07}. For the sake of readers, we give a detailed proof here. Since $\delta(t) \rightarrow 0,~ t \rightarrow 0$, there exists $\delta_0 > 0$ such that $\delta(t)\|L\| \leq r < 1$ for $t \in [0,\delta_0]$. Fix $x \in X_0$. Consider the operator
	\[
	(\widetilde{B}f)(t) = (S_A \diamond Lf)(t), ~~f \in C([0, \delta_0], X),
	\]
	which is contractive and $(\widetilde{B}V(\cdot)x)(t) = (\mathcal{B}V)(t)x$. Thus, we have a unique $W(\cdot)x \in C([0,\delta_0], X)$ such that
	\[
	\widetilde{B}W(\cdot)x + V(\cdot)x = W(\cdot)x,
	\]
	and
	\begin{equation}\label{jssss}
		W(t)x = \sum\limits_{n=0}^{\infty}(\widetilde{B}^nV(\cdot)x)(t) = \sum\limits_{n=0}^{\infty}\mathcal{B}^n(V)(t)x.
	\end{equation}
	Obviously, by the uniqueness, $ x \mapsto W(t)x $ is linear for each $ t $.
	In addition, there is $\widetilde{M} > 0$ such that
	\[
	\|(\widetilde{B}^mV(\cdot)x)(t) - (\widetilde{B}^nV(\cdot)x)(t)\| \leq \frac{r^n}{1-r}\sup_{t\in[0,\delta_0]}\|(S_A \diamond LV)(t)x - V(t)x\| \leq \frac{r^n}{1-r}\widetilde{M}\|x\|,
	\]
	whenever $m \geq n$, which implies \eqref{jssss} is uniformly convergent on $[0,\delta]$ in the uniform operator topology. Therefore $W \in C_s([0,\delta_0], \mathcal{L}(X_0, X))$.
	Next, there is $\widetilde{W}(\cdot)$ satisfying
	\[
	(S_A \diamond L\widetilde{W})(s) + T_{A_0}(s)(S_A \diamond LW)(\delta_0) + V(s+\delta_0) = \widetilde{W}(s), ~ s \in [0, \delta_0].
	\]
	(That is we take $T_{A_0}(\cdot)(S_A \diamond LW)(\delta_0) + V(\cdot+\delta_0)$ as the initial function.) Let
	\[
	W(t) = \widetilde{W}(t - \delta_0), ~t \in [\delta_0, 2\delta_0].
	\]
	Then,
	\[
	W(t) = (S_A \diamond LW(\delta_0+\cdot))(t-\delta_0) + T_{A_0}(t-\delta_0)(S_A \diamond LW)(\delta_0) + V(t) = (\mathcal{B}W)(t) + V(t).
	\]
	So $W(\cdot)$ is well defined on $[\delta_0, 2\delta_0]$ and satisfies \eqref{equ:fix}. Since $W(\cdot)$ is uniquely constructed in this way, this completes the proof.
\end{proof}

It was shown in \cite{MR07, Thi08} that MR operators (quasi Hille--Yosida operators) are stable under the bounded perturbation.
\begin{thm}[\cite{MR07, Thi08}]\label{thm:mrper}
	Let $A$ be an MR operator (resp. $p$-quasi Hille--Yosida operator). For any $L \in \mathcal{L}(X_0, X)$, $A+L$ is still an MR operator (resp. $p$-quasi Hille--Yosida operator), and the following fixed equations hold.
	\begin{align*}
		T_{(A+L)_0} & = T_{A_0} + S_A \diamond LT_{(A+L)_0} \\
		& = T_{A_0} + S_A \diamond LT_{A_0} + S_A \diamond LS_A \diamond LT_{(A+L)_0} \\
		& = \cdots, \\
		S_{A+L} & = S_A + S_A \diamond LS_{A + L} = \cdots.
	\end{align*}
\end{thm}

Set
\begin{equation}\label{symbol}
	S_0 := T_{A_0},~ R_0 := T_{(A+L)_0},~ S_n := \mathcal{B}^n(T_{A_0}),~R_n := \mathcal{B}^n(T_{(A+L)_0}), ~n = 1, 2, \cdots.
\end{equation}
Combining with \autoref{lem:fix}, we have
\begin{align}
	S_{n+1} & = S_A \diamond LS_n \label{equ:sr1}\\
	R_k & = \sum_{n=k}^{\infty} S_n, \label{equ:sr2}\\
	R_k - S_k & = \mathcal{B} \circ R_k, \label{equ:sr3}\\
	T_{(A+L)_0} & = \sum_{n=0}^{k-1}S_n + R_k, \label{equ:sr4}\\
	& = \sum_{n=0}^{\infty}S_n, \label{equ:sr5}
\end{align}
where \eqref{equ:sr2} and \eqref{equ:sr5} are uniformly convergent on finite interval in the uniform operator topology.
The following is a result about $S_n$, which is similar as Dyson-Phillips series (\cite{EN00}).
\begin{lem} \label{lem:srelation}
	$S_n(t+s) = \sum\limits_{k=0}^{n}S_k(t)S_{n-k}(s).$
\end{lem}
\begin{proof}
	The case $n=0$ is clear. Consider
	\begin{align*}
		S_{n+1}(t+s) & = [S_A \diamond LS_n(\cdot)](t+s) \\
		& = T_{A_0}(t)[S_A \diamond LS_n(\cdot)](s) + [S_A \diamond LS_n(s + \cdot)](t) \quad \\
		& = T_{A_0}(t)S_{n+1}(s) + [S_A \diamond L \sum_{k=0}^{n}S_k(\cdot)S_{n-k}(s)](t) \quad \\
		& = T_{A_0}(t)S_{n+1}(s) + [\sum_{k=0}^{n}S_{k+1}(\cdot)S_{n-k}(s)](t) \\
		& = \sum_{k=0}^{n+1}S_k(t)S_{n+1-k}(s),
	\end{align*}
	the second equality being a consequence of \eqref{equ:sf2}, and the third one by induction.
\end{proof}

Using the above lemma, we obtain the following corollary which generalizes the corresponding case of $C_0$ semigroups; \autoref{cor:DP} \eqref{DPee} was also given in \cite[Theorem 3.2]{Bre01} in the context of $C_0$ semigroups.
\begin{cor}\label{cor:DP}
	\begin{enumerate}[(a)]
		\item If $S_0, S_1, \cdots, S_n$ are norm continuous (resp. compact, differentiable) at $t_0 > 0$, then $S_n(t)$ is norm continuous (resp. compact, differentiable) for all $t \geq t_0$.
		\item If $S_0, S_1, \cdots, S_n$ are compact at $t_0 > 0$, then $S_n$ is norm continuous on $(t_0, \infty)$.
		\item \label{DPcc} If $ S_n(t) $ is compact for all $ t > t_1 $ ($ \geq 0 $), then $ S_n $ is norm continuous on $ (t_1, \infty) $.
		\item \label{DPdd} If $ R_n(t) $ is compact for all $ t > t_1 $ ($ \geq 0 $), then $ R_n $ is norm continuous on $ (t_1, \infty) $.
		\item \label{DPee} $ S_n(t) $ is compact for all $ t > 0 $ if and only if $ S_n $ is norm continuous on $ (0, \infty) $ and $ (R(\lambda+ i\mu, A) L)^n R(\lambda + i \mu, A_0) $ is compact for all $ \mu \in \mathbb{R} $ and for some/all large $ \lambda > 0 $.
	\end{enumerate}
\end{cor}
\begin{proof}
	(a) and (b) are direct consequences of \autoref{lem:srelation}.

	(c) Let $ t_0 > t_1 $. Take a constant $ M > 0 $ such that $ \sup_{t \in [0, 2t_0]}\|S_k(t)\| \leq M $, $ k = 0, 1, 2,\cdots,n $. For any small $ \varepsilon > 0 $, by \autoref{lem:MR} \eqref{MRaa}, we know there is $ 0 < \delta < t_0 $ such that $ t_0 - \delta > t_1 $ and $ \sup_{\sigma\in [0, 2\delta]}\|S_k(\sigma)\| \leq \varepsilon / M $. So
	\[
	\sup_{\sigma\in [0, 2\delta]} \sum_{k=1}^{n}\|S_k(\sigma)S_{n-k}(t_0 - \delta)\| \leq \varepsilon.
	\]
	Now for $ 0 < \delta' < \delta $, we see
	\begin{align*}
	& \|S_{n}(t_0) - S_{n}(t_0\pm\delta')\| \\
	= &~ \|(T_{A_0}(\delta) - T_{A_0}(\delta\pm\delta'))S_{n}(t_0 - \delta) + \sum_{k=1}^{n} (S_k(\delta) - S_k(\delta\pm\delta')) S_{n-k}(t_0 - \delta) \| \\
	\leq &~ \|(T_{A_0}(\delta) - T_{A_0}(\delta\pm\delta'))S_{n}(t_0 - \delta)\| + 2\varepsilon.
	\end{align*}
	If we take $ \delta' $ sufficiently small, then $ \|(T_{A_0}(\delta) - T_{A_0}(\delta\pm\delta'))S_{n}(t_0 - \delta)\| \leq \varepsilon $ as $ S_{n}(t_0 - \delta) $ is compact, and so $ \|S_{n}(t_0) - S_{n}(t_0\pm\delta')\| \leq 3\varepsilon $. This shows $ S_{n} $ is norm continuous at $ t_0 > 0 $.

	(d) This is very similar as $ S_n $ by using the following equality:
	\begin{multline*}
	\|R_{n}(t_0) - R_{n}(t_0\pm\delta')\| = \|(T_{A_0}(\delta) - T_{A_0}(\delta\pm\delta'))R_{n}(t_0 - \delta) \\
	+ (S_{A} \diamond LR_{n-1}(\cdot+t_0 - \delta)) (\delta) - (S_{A} \diamond LR_{n-1}(\cdot+t_0 - \delta)) (\delta\pm\delta') \|.
	\end{multline*}

	(e) This follows from the following \autoref{slem:L} and conclusion (c) (letting $ t_1 = 0 $), since $ (R(\lambda+ i\mu, A) L)^n R(\lambda + i \mu, A_0) = \mathcal{L} (S_n) (\lambda + i \mu) $ (the Laplace transform of $ S_n $ at $ \lambda + i \mu $).
	The proof is complete.
\end{proof}
\begin{slem}\label{slem:L}
	For norm continuous $ F: \mathbb{R}_+ \to \mathcal{L}(X_0, X) $ such that $ \|F(t)\| \leq C_1e^{\omega t} $ for all $ t > 0 $ and some constants $ C_1 \geq 1, \omega \in \mathbb{R} $, $ F(t) $ is compact for all $ t > 0 $ if and only if $ \mathcal{L}(F)(\lambda + i \mu) $ (the Laplace transform of $ F $ at $ \lambda + i \mu $) is compact for all $ \mu \in \mathbb{R} $ and some/all $ \lambda > \omega $.
\end{slem}
\begin{proof}
	If $ F(t) $ is compact for all $ t > 0 $, then
	\[
	\mathcal{L}(F)(\lambda + i \mu) = \int_{0}^{\infty} e^{-(\lambda + i \mu)t} F(t) ~\mathrm{d} t = \lim_{r \to \infty} \int_{0}^{r} e^{-(\lambda + i \mu)t} F(t) ~\mathrm{d} t ,
	\]
	is compact. And if $ \mathcal{L}(F)(\lambda + i \mu) $ is compact for all $ \mu \in \mathbb{R} $ and some $ \lambda > \omega $, then by the complex inversion formula of Laplace transform (see, e.g., \cite[Theorem 2.3.4]{ABHN11}), one gets
	\[
	\int_{0}^{t} e^{-\omega s} F(s) ~\mathrm{d} s = \lim_{\mu  \to \infty} \frac{1}{2\pi i} \int_{\lambda - \omega - i\mu}^{\lambda - \omega + i \mu} e^{rt} \frac{\mathcal{L}(F)(r+\omega)}{r} ~\mathrm{d} r,
	\]
	where the limit is uniform for $ t $ belonging to compact intervals and exists in the uniform operator topology. This shows that $ \int_{0}^{t} e^{-\omega s} F(s) ~\mathrm{d} s $ is compact and consequently $ F(t) $ is compact for all $ t > 0 $. The proof is complete.
\end{proof}
\noindent\emph{Problem}: for Hilbert space $ X $, is it true that if $ \|(R(\lambda+ i\mu, A) L)^n R(\lambda + i \mu, A_0)\| \to 0 $ as $ |\mu| \to \infty $, then $ S_{n-2} $ is norm continuous? (See \cite{Bre01}.) When $ LR(\lambda, A_0) = R(\lambda, A)L $, the proof is very similar.

\subsection{Critical spectrum and essential spectrum}\label{spec}
In this subsection, we recall some known results about spectral theory. For a complex Banach algebra $ Y $, $\sigma(x)$ denotes the spectrum of $x \in Y$, i.e.,
\[
\sigma(x) = \{\lambda \in \mathbb{C}:~ \lambda - x ~\text{is not invertible in $ Y $}\},
\]
and $r(x)$ denotes the spectral radius of $x$, i.e.,
\[
r(x) = \sup\{|\lambda|:~\lambda \in \sigma(x)\}.
\]

Critical spectrum was discovered by Nagel and Poland \cite{NP00}. Let $T$ be a $C_0$ semigroup on a Banach space $X$, whose generator is $B$. Set
\[
l^\infty(X) := \{ (x_n): \|(x_n)\|_\infty \triangleq \sup\limits_n \|x_n\| < \infty \}.
\]
$T$ can be naturally extended to $l^\infty(X)$ by
\[
\widetilde{T}(t)(x_n) := (T(t)x_n), ~~(x_n) \in l^\infty(X).
\]
Then, $\|\widetilde{T}(t)\| = \|T(t)\|$. Define
\begin{align*}
	l_T^\infty(X) & := \{(x_n) \in l^\infty(X): \sup\limits_n\|T(t)x_n - x_n\| \rightarrow 0, ~\text{as}~ t \rightarrow 0^+ \} \\
	& = \{(x_n) \in l^\infty(X): \widetilde{T}(t)(x_n) \rightarrow (x_n), ~\text{as}~ t \rightarrow 0^+ \} \\
	& = \{(x_n) \in l^\infty(X): \sup\limits_{n \rightarrow \infty}\|T(t)x_n - x_n\| \rightarrow 0, ~\text{as}~ t \rightarrow 0^+ \}.
\end{align*}
$l_T^\infty(X)$ is the closed space of strong continuity of $\widetilde{T}$. Note that $\widetilde{T}(t)l_T^\infty(X) \subset l_T^\infty(X)$, for each $t \geq 0 $. So $\widetilde{T}$ can induce $\widehat{T}$, defined on the quotient space $l^\infty(X) / l_T^\infty(X)$ as
\[
\widehat{T}(t)[(x_n)] := [\widetilde{T}(t)(x_n)], ~~\text{where}~[(x_n)] := (x_n) + l_T^\infty(X).
\]

In general, if $F \in \mathcal{L} (X)$, define $\widetilde{F}$ on $l^\infty(X)$ as
\[
\widetilde{F}(x_n) := (F(x_n)),
\]
with $\|\widetilde{F}\| = \|F\|$. If $\widetilde{F} (l_T^\infty(X)) \subset l_T^\infty(X)$, then $\widehat{F}$ can be defined on $l^\infty(X) / l_T^\infty(X)$ by
\[
\widehat{F}[(x_n)] := [\widetilde{F}(x_n)],
\]
with $\|\widehat{F}\| \leq \|F\|$. If $G \in \mathcal{L}(X)$ and $\widetilde{G} (l_T^\infty(X)) \subset l_T^\infty(X)$, then
\[
\widetilde{F \circ G} = \widetilde{F} \circ \widetilde{G} ,~~ \widehat{F \circ G} = \widehat{F} \circ \widehat{G}.
\]
Note that $l^\infty(X) / l_T^\infty(X)$ is a Banach algebra. Now we give the following definitions.
\begin{defi}[\cite{NP00}]\label{def:crit}
	For a $C_0$ semigroup $T$, we call
	\[
	\sigma_{\mathrm{crit}}(T(t)) := \sigma(\widehat{T}(t)),
	\]
	the critical spectrum of $T(t)$,
	\[
	r_{\mathrm{crit}}(T(t)) := r(\widehat{T}(t)),
	\]
	the critical spectral radius of $T(t)$, and
	\begin{align*}
		\omega_{\mathrm{crit}}(T) & := \inf \{ \omega \in \mathbb{R}: \exists M \geq 0,~ \|\widehat{T}(t)\| \leq Me^{\omega t}, \forall t \geq 0 \}\\
		& = \lim_{t \rightarrow \infty} \frac{1}{t} \ln \|\widehat{T}(t)\|
		= \inf_{t >0} \frac{1}{t} \ln \|\widehat{T}(t)\|
		= \frac{1}{t_0}\ln r_{\mathrm{crit}}(T(t_0)) ~ (t_0 >0),
	\end{align*}
	the critical growth bound of semigroup $T$.
\end{defi}
With these definitions, the following theorem holds.
\begin{thm}[\cite{NP00}]\label{thm:crit}
	For a $C_0$ semigroup $T$ with generator $B$, the following statements hold.
	\begin{enumerate}[(a)]
		\item $\sigma_{\mathrm{crit}}(T(t)) \subset \sigma(T(t))$.
		\item (Partial spectral mapping theorem) For each $t > 0$,
		\[
		\sigma(T(t))\backslash \{0\} = e^{t\sigma(B)} \cup \sigma_{\mathrm{crit}}(T(t)) \backslash \{0\},
		\]
		and
		\[
		\sigma(T(t)) \cap \{r \in \mathbb{C}: |r| > r_{\mathrm{crit}}(T(t))\} = e^{t\sigma(B)} \cap \{r \in \mathbb{C}: |r| > r_{\mathrm{crit}}(T(t))\}.
		\]
		\item $\omega(T) = \max\{ s(B), \omega_{\mathrm{crit}}(T) \}$,
	\end{enumerate}
	where we denote by $\omega(T)$ the growth bound of T, and $s(B)$ the spectral bound of $B$, i.e.,
	\begin{equation}\label{equ:bound}
	\omega(T) = \inf \{ \omega \in \mathbb{R}: \exists M \geq 0,~ \|T(t)\| \leq Me^{\omega t}, \forall t \geq 0 \}, ~s(B) = \sup\{\mathrm{Re}\lambda: \lambda \in \sigma(B) \}.
	\end{equation}
\end{thm}

See \cite{Bla01} for a different but equivalent definition of critical spectrum.
We refer to \cite{NP00, BNP00, Sbi07} for more details and applications on critical spectrum. Next, we turn to the essential spectrum. We follow the presentation of \cite{GGK90}. Consider the Calkin algebra $\mathcal{L}(X)/ \mathcal{K}(X)$, where $\mathcal{K}(X)$ denotes the closed ideal of bounded compact linear operators on $X$. Define $\overline{T}$ on $\mathcal{L}(X)/ \mathcal{K}(X)$ by
\[
\overline{T}(t) := [T(t)] := T(t) + \mathcal{K}(X).
\]
Now we have the following definitions.
\begin{defi}\label{def:ess}
	For a $C_0$ semigroup $T$, we call
	\[
	\sigma_{\mathrm{ess}}(T(t)) := \sigma(\overline{T}(t)),
	\]
	the essential spectrum of $T(t)$,
	\[
	r_{\mathrm{ess}}(T(t)) := r(\overline{T}(t)),
	\]
	the essential spectral radius of $T(t)$ and
	\begin{align*}
		\omega_{\mathrm{ess}}(T) & := \inf \{ \omega \in \mathbb{R}: \exists M \geq 0,~ \|\overline{T}(t)\| \leq Me^{\omega t}, \forall t \geq 0 \}\\
		& = \lim_{t \rightarrow \infty} \frac{1}{t} \ln \|\overline{T}(t)\|
		= \inf_{t >0} \frac{1}{t} \ln \|\overline{T}(t)\|
		= \frac{1}{t_0}\ln r_{\mathrm{ess}}(T(t_0)) ~ (t_0 >0),
	\end{align*}
	the essential growth bound of semigroup $T$.
\end{defi}
\begin{thm}[\cite{NP00, EN00}]\label{thm:ess}
	For a $C_0$ semigroup $T$ with generator $B$, the following statements hold.
	\begin{enumerate}[(a)]
		\item $\sigma_{\mathrm{ess}}(T(t)) \subset \sigma_{\mathrm{crit}}(T(t)) \subset \sigma(T(t))$. Thus, $\omega_{\mathrm{ess}}(T) \leq \omega_{\mathrm{crit}}(T) \leq \omega(T)$, $\omega(T) = \max\{ s(B), \omega_{\mathrm{crit}}(T) \}$.
		\item $\forall \gamma > \omega_{\mathrm{ess}}(T)$, the set $\{\lambda \in \sigma(B): \mathrm{Re}\lambda \geq \gamma \}$ is finite. And if it is not empty, then its elements are the finite-order poles of $R(\cdot, B)$, in particular the point spectrum of $B$.
	\end{enumerate}
\end{thm}

For other definitions of essential spectrum, see \cite{Dei85}. Although different definitions of essential spectrum may not be coincided, all the essential spectral radiuses are equal.

\section{General arguments}

From now on, we assume that $A$ is an MR operator, $V \in C_s([0,\infty), \mathcal{L}(X_0, X))$. Then, $A_0 := A_{X_0}$ generates an $C_0$ semigroup $T_{A_0}$ in $X_0$, and $A$ generates an integrated semigroup $S_A$, where $X_0 := \overline{D(A)}$. In this section, we consider the regularity of $S_A \diamond V$, when $T_{A_0}$ or $V$ has higher regularity. Note that $t \mapsto (S_A \diamond V)(t)$ is norm continuous at zero (see \autoref{lem:MR} \eqref{MRaa}).
\begin{lem}\label{lem:NC}
	\begin{enumerate}[(a)]
		\item If $T_{A_0}$ is norm continuous on $(0,\infty)$, then $S_A \diamond V$ is norm continuous on $(0,\infty)$ (hence $[0,\infty)$). More generally, if $L \in \mathcal{L}(X_0, X)$ and $LT_{A_0}$ is norm continuous on $(0, \infty)$, so is $LS_{A} \diamond V$.
		\item If $V$ is norm continuous on $(0,\infty)$, so is $S_A \diamond V$.
		\item If $T_{A_0}$ in norm continuous on $(\alpha, \infty)$, and $V$ is norm continuous on $(\beta, \infty)$, then $S_A \diamond V$ is norm continuous on $(\alpha + \beta, \infty)$.
	\end{enumerate}
\end{lem}
\begin{proof}
	Let $t > s > 0$ in $U(t_0;\eta) = \{ r \in \mathbb{R}: |r - t_0| < \eta \}$ (a neighborhood of $t_0 > 0$), and $0 < \varepsilon < s$, where $t_0 > \eta > 0$.

	(a) By \eqref{equ:sf2}, we have
	\begin{align*}
		L(S_A \diamond V)(t) - L(S_A \diamond V)(s) & = LT_{A_0}(t-s)(S_A \diamond V)(s) + L(S_A \diamond V(s + \cdot))(t-s) - L(S_A \diamond V)(s)\\
		& = (LT_{A_0}(t-s + \varepsilon ) - LT_{A_0}(\varepsilon))(S_A \diamond V)(s-\varepsilon) + \\
		&  \qquad L(T_{A_0}(t-s) - I)(S_A \diamond V(s - \varepsilon + \cdot))(\varepsilon) + L(S_A \diamond V(s + \cdot))(t-s) \\
		& := (LT_{A_0}(t-s + \varepsilon ) - LT_{A_0}(\varepsilon))(S_A \diamond V)(s-\varepsilon) + R,
	\end{align*}
	where
	\[
	R := L(T_{A_0}(t-s) - I)(S_A \diamond V(s - \varepsilon + \cdot))(\varepsilon) + L(S_A \diamond V(s + \cdot))(t-s),
	\]
	and
	\[
	\|R\| \leq \widetilde{M}(\delta(\varepsilon) + \delta(t-s)),
	\]
	for some constant $\widetilde{M} > 0$. Let $\varepsilon$, $\eta$ be sufficiently small ($\eta$ depending on $\varepsilon$), then $\|L(S_A \diamond V)(t) - L(S_A \diamond V)(s)\|$ can be sufficiently small, which shows that $LS_A \diamond V$ is norm continuous at $t_0$.

	(b) By \eqref{equ:sf2}, we see
	\begin{align*}
		&\|(S_A \diamond V)(t) - (S_A \diamond V)(s)\| \\
		= & \|T_{A_0}(s-\varepsilon)(S_A \diamond V)(t-s+\varepsilon) + (S_A \diamond V(t-s+\varepsilon + \cdot))(s-\varepsilon) \\
		& - T_{A_0}(s-\varepsilon)(S_A \diamond V)(\varepsilon) - (S_A \diamond V(\varepsilon + \cdot))(s-\varepsilon)\| \\
		\leq & \widetilde{M}(\delta(t-s+\varepsilon) + \delta(\varepsilon)) + \widetilde{M}\sup_{r \in [0,s-\varepsilon]} \|V(t-s+\varepsilon+r)-V(\varepsilon+r)\|,
	\end{align*}
	for some constant $\widetilde{M} > 0$. Since $V$ is uniformly continuous on $[\varepsilon, 2t_0]$, the last inequality can be sufficiently small.

	(c) When $t > \alpha + \beta$, by \eqref{equ:sf2}, we get
	\begin{equation}\label{equ:general}
		(S_A \diamond V)(t) = T_{A_0}(t-\beta)(S_A \diamond V)(\beta) + (S_A \diamond V(\beta + \cdot))(t-\beta).
	\end{equation}
	The first term of the right side is norm continuous by norm continuity of $T_{A_0}$,  and the second is norm continuous by (b). The proof is complete.
\end{proof}
For $ W \in C_s([0,\infty), \mathcal{L}(X_0, X)) $, we say $ W $ is compact on $ (a,b) $ if for each $ t \in (a,b) $, $ W(t) $ is compact.
\begin{lem}\label{lem:CP}
	\begin{enumerate}[(a)]
		\item If $T_{A_0}$ is compact on $(0,\infty)$, so is $S_A \diamond V$. More generally, if $L \in \mathcal{L}(X_0, X)$ and $LT_{A_0}$ is compact on $(0,\infty)$, so is $LS_A \diamond V$.
		\item If $V$ is compact and norm continuous on $(0,\infty)$, so is $S_A \diamond V$.
		\item If $T_{A_0}$ is compact on $(\alpha, \infty)$, and $V$ is compact and norm continuous on $(\beta, \infty)$, then $S_A \diamond V$ is compact and norm continuous on $(\alpha + \beta, \infty)$.
	\end{enumerate}
\end{lem}
\begin{proof}
	(a) For $ t > \varepsilon > 0 $, by
	\[
	L(S_A \diamond V)(t) = LT_{A_0}(\varepsilon)(S_A \diamond V)(t - \varepsilon) + L(S_A \diamond V(t-\varepsilon+\cdot))(\varepsilon),
	\]
	we see $LT_{A_0}(\varepsilon)(S_A \diamond V)(t - \varepsilon)$ converges to $L(S_A \diamond V)(t)$ in the operator norm topology, as $\varepsilon \rightarrow 0^+$. Since $ LT_{A_0}(\varepsilon) $ is compact, the result follows.

	(b) Since $V$ is norm continuous on $(0,\infty)$, so is $S_A \diamond V$ (\autoref{lem:CP} (b)). Hence,
	\[
	(S_A \diamond V)(t) = \lim_{h \rightarrow 0} \frac{1}{h} \int_0^tT_{A_0}(t-s)S_A(h)V(s)~\mathrm{d}s,
	\]
	where the convergence is in the uniform operator topology; see the proof of equation \eqref{equ:sf3}. Since $V(t)$ is compact for all $t>0$, $\int_0^tT_{A_0}(t-s)S_A(h)V(s)~\mathrm{d}s$ is compact (see, e.g., \cite[Theorem C.7]{EN00}), yielding $(S_A \diamond V)(t)$ is compact.

	(c) Use \eqref{equ:general} and (b). The proof is complete.
\end{proof}

The following result is a simple relation of differentiability in strong topology and uniform operator topology; the proof is omitted.
\begin{lem}\label{lem:reldiff}
	Let $f:~ I \rightarrow \mathcal{L}(X,Y)$, where $I$ is an interval of $\mathbb{R}$.
	\begin{enumerate}[(a)]
		\item If $t \mapsto f(t)x$ is continuously differentiable for each $x \in X$, then $t \mapsto f(t)$ is norm continuous.
		\item If $t \mapsto f(t)x$ is differentiable for each $x \in X$, let $G(t)x = f'(t)x$, $\forall x \in X$. If $t \mapsto G(t)$ is norm continuous, then $t \mapsto f(t)$ is norm continuously differentiable.
	\end{enumerate}
\end{lem}
\begin{lem}\label{lem:diff}
	\begin{enumerate}[(a)]
		\item If $V$ is strongly continuously differentiable (resp. norm continuously differentiable) on $[0,\infty)$, and $V(0) = 0$, so is $S_A \diamond V$. Moreover, $S_A \diamond V = S_A * V'$.
		\item If $T_{A_0}$ is strongly continuously differentiable on $[\alpha, \infty)$ and $V$ is strongly continuously differentiable on $[\beta, \infty)$, then $S_A \diamond V$ is strongly continuously differentiable on $[\alpha + \beta, \infty)$.
	\end{enumerate}
\end{lem}
\begin{proof}
	(a) Since
	\[
	(S_A \diamond V)(t) = S_A(t)V(0) + \int_0^tS_A(s)V'(t-s)~\mathrm{d}s = S_A * V',
	\]
	by the assumption of $A$, it shows $S_A * V$ is strongly continuously differentiable (see \autoref{lem:MR} (c)). Thus, $S_A \diamond V$ is strongly continuously differentiable. For the case of norm continuously differentiable, by \autoref{lem:NC} (b), $S_A \diamond V'$ is norm continuous. The result follows from \autoref{lem:reldiff} (b).

	(b) For $t \geq \alpha + \beta$,
	\begin{align*}
		(S_A \diamond V)(t) & = T_{A_0}(t-\beta)(S_A \diamond V)(\beta) + (S_A \diamond V(\beta + \cdot))(t-\beta) \\
		& = T_{A_0}(t-\beta)(S_A \diamond V)(\beta) + S_A(t-\beta)V(\beta) + S_A * V(\beta + \cdot)'.
	\end{align*}
	The first term of the right side is strongly continuously differentiable. By \autoref{partInt} (e), $S_A$ is strongly continuously differentiable on $[\alpha, \infty)$. By the assumption on $A$, the third term is also strongly continuously differentiable (see \autoref{lem:MR} (c)). This completes the proof.
\end{proof}

\section{Regularity of perturbed semigroups}\label{regper}
If $A$ is an MR operator, $L \in \mathcal{L}(X_0, X)$, then $A+L$ is also an MR operator. In this section, we study the regularity properties of the perturbed semigroup $T_{(A+L)_0}$, generated by $(A+L)_0$ in $X_0$. We make some regularity assumptions similar as \cite{NP98, BMR02} to ensure $T_{(A+L)_0}$ preserves the regularity of $T_{A_0}$.
\begin{lem}\label{lem:regfix}
	Suppose that $V \in C_s([0,\infty), \mathcal{L}(X_0, X))$ has one of the following properties,
	\begin{enumerate}[(a)]
		\item norm continuity on $(0,\infty)$;
		\item norm continuity and compactness on $(0,\infty)$;
		\item $V(0) = 0$ and strongly continuous differentiability on $[0,\infty)$.
	\end{enumerate}
	Then, the solution $W$ of the equation
	\[
	W = V + \mathcal{B}(W),
	\]
	has the corresponding property.
\end{lem}
\begin{proof}
	If $V$ has the property of (a) (resp. (b)), by \autoref{lem:NC} (resp. \autoref{lem:CP}), $\mathcal{B}^k(V)$ has the property of (a) (resp. (b)) for $k = 0,1,2,\cdots$. By \autoref{lem:fix}, $W = \sum\limits_{k=0}^{\infty}\mathcal{B}^k(V)$ is internally closed uniformly norm convergent. So $W$ has property of (a) (resp. (b)).

	If $V(0) = 0$ and $V$ is strongly continuously differentiable on $[0,\infty)$, then by \autoref{lem:diff} (a), $[\mathcal{B}(V)]' = \mathcal{B}(V')$, and $[\mathcal{B}^k(V)]' = \mathcal{B}^k(V')$ $(k \geq 1)$. (Note that $\mathcal{B}^k(V)(0) = 0$ for $k \geq 1$.) Since $\sum\limits_{k=0}^{\infty}\mathcal{B}^k(V)$ and $\sum\limits_{k=0}^{\infty}(\mathcal{B}^k(V))' = \sum\limits_{k=0}^{\infty}\mathcal{B}^k(V')$ are internally closed uniformly convergent, we have $W' = \sum\limits_{k=0}^{\infty}\mathcal{B}^k(V')$. That is $W' = V' + \mathcal{B}(W')$, which completes the result.
\end{proof}

Next, we consider the relation between the regularity of $S_n$, $R_n$ (see \eqref{symbol}).
\begin{cor}\label{corr:sr}
	The following are equivalent.
	\begin{enumerate}[(a)]
		\item $S_n$ has property $P$;
		\item $S_k$ has property $P$, $\forall k \geq n$;
		\item $R_n$ has property $P$;
		\item $R_k$ has property $P$, $\forall k \geq n$.
	\end{enumerate}
	The property $P$ stands for norm continuity, or compactness on $(0, \infty)$. If $n \geq 1$, $P$ also stands for strongly continuous differentiability on $[0, \infty)$.
\end{cor}
\begin{proof}
	(a) $\Rightarrow$ (b) By \autoref{lem:NC} (or \autoref{lem:CP}, or \autoref{lem:diff}) and \eqref{equ:sr1}.
	(b) $\Rightarrow$ (c) By \eqref{equ:sr1} and \autoref{lem:regfix}.
	(c) $\Rightarrow$ (d) The same reason as (a) $\Rightarrow$ (b).
	(d) $\Rightarrow$ (a) By \eqref{equ:sr3} and \autoref{lem:NC} (or \autoref{lem:CP}, or \autoref{lem:diff}). Here note that if $ S_{n} $ (or $ R_{n} $) is compact on $ (0, \infty) $, then it is norm continuous on $ (0, \infty) $ by \autoref{cor:DP} \eqref{DPcc} \eqref{DPdd}.
\end{proof}

\begin{thm}
	Let $A$ be an MR operator, $L \in \mathcal{L}(X_0, X)$. Then, the following statements hold.
	\begin{enumerate}[(a)]
		\item $T_{A_0}$ is immediately norm continuous if and only if $T_{(A+L)_0}$ is immediately norm continuous.
		\item $T_{A_0}$ is immediately compact if and only if $T_{(A+L)_0}$ is immediately compact.
	\end{enumerate}
\end{thm}
\begin{proof}
	It's a direct result of \autoref{corr:sr} for the case $n=0$. Note that $C_0$ semigroup is immediately compact, then it is also immediately norm continuous (see, e.g., \cite[Lemma II.4.22]{EN00}).
\end{proof}
\begin{thm}\label{thm:perreg}
	Let $A$ be an MR operator, $L \in \mathcal{L}(X_0, X)$, $\alpha \geq 0$. Then, the following statements hold.
	\begin{enumerate}[(a)]
		\item If $T_{A_0}$ is norm continuous on $(\alpha, \infty)$, and there exists $S_n$ which is norm continuous on $(0,\infty)$, then $T_{(A+L)_0}$ is norm continuous on $(n\alpha, \infty)$.
		\item If $T_{A_0}$ is compact on $(\alpha, \infty)$, and there exists $S_n$ which is compact on $(0,\infty)$, then $T_{(A+L)_0}$ is compact (and norm continuous) on $(n\alpha, \infty)$.
		\item If $T_{A_0}$ is strongly continuously differentiable on $[\alpha, \infty)$, and there exists $S_n$ ($n \geq 1$) which is strongly continuously differentiable on $[0,\infty)$, then $T_{(A+L)_0}$ is strongly continuously differentiable on $[n\alpha, \infty)$.
	\end{enumerate}
\end{thm}
\begin{proof}
	If $T_{A_0}$ is norm continuous (resp. compact (hence norm continuous), or differentiable) on $(\alpha, \infty)$, then $S_m$ is norm continuous (resp. compact and norm continuous, or strongly continuously differentiable) on $(m\alpha, \infty)$ by \autoref{lem:NC} (c) (resp. \autoref{lem:CP} (c), or \autoref{lem:diff} (b)). Combining with \autoref{corr:sr} and \eqref{equ:sr4}, we obtain the results. Here note that if $S_n$ is compact on $(0,\infty)$, then it is also norm continuous on $(0,\infty)$ by \autoref{cor:DP} \eqref{DPcc}.
\end{proof}

In applications, we need to calculate the $S_n$ to show it has higher regularity, see some examples in \cite{BNP00, BMR02}. Here we give some conditions such that $S_n$ has higher regularity. The following result can be obtained directly by \autoref{lem:NC} (a), \autoref{lem:CP} (a) and \autoref{corr:sr}.
\begin{cor}\label{corr:s1} Let $A$ be an MR operator, $L \in \mathcal{L}(X_0, X)$.
	\begin{enumerate}[(a)]
		\item If $LT_{A_0}$ is norm continuous on $(0, \infty)$, then $S_1$ and $T_{(A+L)_0} - T_{A_0}$ are norm continuous on $(0, \infty)$. Particularly, in this case, $T_{A_0}$ is eventually norm continuous if and only if $T_{(A+L)_0}$ is eventually norm continuous.
		\item If $LT_{A_0}$ is norm continuous and compact on $(0, \infty)$, then $S_1$ and $T_{(A+L)_0} - T_{A_0}$ are norm continuous and compact on $(0, \infty)$. Particularly, in this case, $T_{A_0}$ is eventually compact if and only if $T_{(A+L)_0}$ is eventually compact.
	\end{enumerate}
\end{cor}

For $p$-quasi Hille--Yosida operator $ A $, the compactness of $LT_{A_0}$ can make $S_2$ be higher regularity, which is discovered by Ducrot, Liu, Magal \cite{DLM08}; the key of the proof used the following fact: $\forall x^* \in (X_0)^*, x \in X$, $x^*S_A(\cdot)x \in W_{loc}^{1,p'}(0,\infty)$ (as $ x^*S_A(\cdot)x $ is of bounded $ p' $-variation), where $1/{p'} + 1/p = 1$ (see also \autoref{rmk:cc}).
\begin{cor}\label{corr:s2}
	Let $A$ be a quasi Hille--Yosida operator. If $LT_{A_0}$ is compact on $(0, \infty)$, then $S_2$ is norm continuous and compact on $(0,\infty)$. Particularly, in this case, $T_{A_0}$ is eventually norm continuous (resp. eventually compact) if and only if $T_{(A+L)_0}$ is eventually norm continuous (resp. eventually compact).
\end{cor}
\begin{proof}
	Since $LT_{A_0}$ is compact on $(0, \infty)$, we have $LS_A \diamond LT_{A_0}$ is norm continuous (\cite[Proposition 4.8]{DLM08}) and compact (\autoref{lem:CP} (a)) on $(0, \infty)$. Thus, $S_2$ is norm continuous and compact on $(0, \infty)$ (\autoref{lem:CP} (b)). Next we show that if $LT_{A_0}$ is compact on $(0, \infty)$, so is $LT_{(A+L)_0}$. Since by \autoref{thm:mrper}, we have
	\[
	LT_{(A+L)_0} = LT_{A_0} + LS_A \diamond LT_{(A+L)_0},
	\]
	and since by \autoref{lem:CP} (a), $LS_A \diamond LT_{(A+L)_0}$ is compact on $(0, \infty)$, it yields the compactness of $LT_{(A+L)_0}$. Then, the last statement follows from \autoref{thm:perreg} and the previous argument.
\end{proof}
\noindent\emph{Problem}: is it true that if $ LS_1 $ is compact, then $ S_3 $ is compact and norm continuous?

The above results don't show that whether $T_{(A+L)_0}$ is immediately differentiable if $T_{A_0}$ is immediately differentiable. In fact, though $A$ is a generator of a $C_0$ semigroup, the result may not hold \cite{Ren95}. We need more detailed characterization of differentiability. Set
\[
D_{\beta, c} \triangleq \{\lambda \in \mathbb{C}: \mathrm{Re}\lambda \geq c - \beta \log|\mathrm{Im\lambda}|\},
\]
where $\beta >0,~ c \in \mathbb{R}$. The following deep result about the characterization of differentiability is due to Pazy and Iley, see \cite[Theorem 4.1, Theorem 4.2]{Ile07}.
\begin{thm}[Pazy-Iley]\label{thm:PI}
	Let $B$ be the generator of a $C_0$ semigroup $T_{B}$. The following are equivalent.
	\begin{enumerate}[(a)]
		\item For every $C \in \mathcal{L}(X)$, $T_{B+C}$, generated by $B+C$, is eventually strongly differentiable (resp. immediately strongly differentiable).
		\item There is some (resp. all) $\beta > 0$ , and a constant $k$, such that $D_{\beta, c} \subset \rho(B)$, $\|R(\lambda, B)\| \leq k, ~\forall \lambda \in D_{\beta, c}$, for some $c \in \mathbb{R}$.
		\item There is some (resp. all) $\beta' > 0$ , such that $D_{\beta', c'} \subset \rho(B)$,  $\|R(\lambda, B)\| \rightarrow 0$, as $\lambda \rightarrow \infty$, $\forall \lambda \in D_{\beta', c'}$, for some $c' \in \mathbb{R}$.
	\end{enumerate}
\end{thm}
\begin{rmk}
	That (a) implies (c) was proved by Pazy \cite{Paz68}. Other's equivalence was proved by Iley \cite{Ile07}. See more results on characterization of differentiability in \cite[Section 3]{BK10}.
\end{rmk}
Pazy \cite{Paz68} characterized the generators of eventually and immediately strongly differentiable semigroups as follows.
\begin{thm}[Pazy criterion of differentiability]\label{thm:pz}
	Let $B$ be the generator of the $C_0$ semigroup $T$, $\|T(t)\| \leq M e^{\omega t}$. Then, $T$ is eventually (resp. immediately) strongly differentiable if and only if there is some (resp. all) $\beta > 0$, $c \in \mathbb{R}$, $m \in \mathbb{N}$, $k > 0$, such that $D_{\beta, c} \subset \rho(B)$ and
	\[
	\|R(\lambda, B)\| \leq |\mathrm{Im} \lambda|^m,
	\]
	for all $\lambda \in D_{\beta, c}$, $\mathrm{Re} \lambda \leq \omega$. In this case, $T$ is strongly differentiable on $((m + 2)/\beta, \infty)$.
\end{thm}

Now the question arises: For what an MR operator $A$, is $T_{(A+L)_0}$ eventually differentiable for all $L \in \mathcal{L}(X_0, X)$? By \autoref{thm:PI}, we know $A_0$ at least satisfies the condition of \autoref{thm:PI} (b) or (c) in $X_0$ (i.e., consider the case $L \in \mathcal{L}(X_0)$).

When $\|R(\lambda, A)\|_{\mathcal{L}(X)}\|L\|_{\mathcal{L}(X_0,X)} < 1$, by \autoref{partop}, we have
\[
R(\lambda, (A+L)_0) = R(\lambda, A+L)|_{X_0} = (I - R(\lambda, A)L)^{-1}R(\lambda, A)|_{X_0} = (I - R(\lambda, A)L)^{-1}R(\lambda, A_0),
\]
and so
\begin{equation}\label{equ:gj1}
	\|R(\lambda, (A+L)_0)\|_{\mathcal{L}(X_0)} \leq  \frac{\|R(\lambda, A_0)\|_{\mathcal{L}(X_0)}}{1 - \|R(\lambda, A)\|_{\mathcal{L}(X)}\|L\|_{\mathcal{L}(X_0,X)}}.
\end{equation}

We will use this estimate to characterize the differentiability of $T_{(A+L)_0}$.
\begin{thm}
	Let $A$ be an MR operator satisfying $\limsup\limits_{\lambda \rightarrow \infty} \|\lambda R(\lambda, A)\| < \infty$ (e.g., $A$ is a Hille--Yosida operator). Then, $T_{(A+L)_0}$ is eventually strongly differentiable for all $L \in \mathcal{L}(X_0, X)$ if and only if $A_0$ satisfies the Pazy-Iley condition, i.e., the condition of \autoref{thm:PI} in (b) or (c).
\end{thm}
\begin{proof}
	Only need to consider the sufficiency. Assume $L \neq 0$ and $A_0$ satisfies the condition of \autoref{thm:PI} (b) or (c). Let $\varepsilon = \frac{\|L\|^{-1}}{4}$, $\lambda \in D_{\beta, c}$, $\mu = \omega + |\lambda|$, where $\omega > 0$ is sufficiently large such that (by the assumption $\limsup\limits_{\lambda \rightarrow \infty} \|\lambda R(\lambda, A)\| < \infty$)
	\[
	\|R(\mu, A)\| \leq \frac{M}{\mu} \leq \frac{M}{\omega} < \varepsilon.
	\]
	Then,
	\begin{align*}
		\|R(\lambda, A)\| & = \|(\mu - \lambda)R(\lambda, A_0)R(\mu, A) + R(\mu, A)\| \\
		& \leq M \frac{|\mu - \lambda|}{\mu}\|R(\lambda, A_0)\| + \varepsilon \\
		& \leq M \frac{2|\lambda|+\omega}{|\lambda|+\omega}\|R(\lambda, A_0)\| + \varepsilon.
	\end{align*}
	Let $K$ be sufficiently large, $D_{\beta', c'} \subset D_{\beta, c}$, such that
	\[
	\|R(\lambda, A_0)\| \leq \frac{\varepsilon}{2M},
	\]
	for $\lambda \in D_{\beta', c'}$, $|\lambda| > K$. Hence by \eqref{equ:gj1}, we get
	\[
	\| R(\lambda, (A+L)_0) \|_{\mathcal{L}(X_0)} \leq 2\|R(\lambda, A_0) \|_{\mathcal{L}(X_0)}.
	\]
	The result now follows from \autoref{thm:pz}.
\end{proof}

Next we consider another class of operators. A $C_0$ semigroup $T$ with the generator $B$ is called the \emph{Crandall--Pazy class} if $T$ is strongly immediately differentiable and
\begin{equation}\label{equ:cp1}
	\|BT(t)\| = O(t^{-\alpha}), ~~\text{as}~ t \rightarrow 0^+,
\end{equation}
for some $\alpha \geq 1$. We also call the generator $B$ a \emph{Crandall--Pazy operator}. It was shown in \cite{CP69} that this is equivalent to
\begin{equation}\label{equ:cp2}
	\|R(iy,B)\| = O(|y|^{-\beta}), ~\text{as}~ |y| \rightarrow \infty,
\end{equation}
for some $1 \geq \beta > 0$. Indeed, \eqref{equ:cp1} implies \eqref{equ:cp2} with $\beta = \frac{1}{\alpha}$, and \eqref{equ:cp2} implies \eqref{equ:cp1} for any $\alpha > \frac{1}{\beta}$. See \cite[Theorem 5.3]{Ile07} for a new characterization of the Crandall--Pazy class.
\begin{thm}
	Let $A$ be an MR operator satisfying $\limsup\limits_{\stackrel{\mu \rightarrow \infty}{\mu > \widehat{\omega}}} \|(\mu - \widehat{\omega})^{\frac{1}{p}} R(\mu, A)\| < \infty$, where $p \geq 1, ~\widehat{\omega} \in \mathbb{R}$ (e.g., $A$ is a $p$-quasi Hille--Yosida operator, see \autoref{lem:pHY} in (c)). In addition, $A_0$ is a Crandall--Pazy operator satisfying
	\[
	\|R(iy, A_0)\|_{\mathcal{L}(X_0)} = O(|y|^{-\beta}),~ |y| \rightarrow \infty, ~ \beta > 1- \frac{1}{p}.
	\]
	Then, $(A+L)_0$ is still a Crandall--Pazy operator for any $L \in \mathcal{L}(X_0, X)$.
\end{thm}
\begin{proof}
	Assume $L \neq 0$. Let $K_1 > 0$ be sufficiently large such that for all $|y| > K_1$,
	\[
	\|R(iy, A_0)\| \leq \frac{M}{|y|^\beta}, ~~\|R(\mu, A)\| \leq \frac{M}{|\mu - \widehat{\omega}|^{\frac{1}{p}}},
	\]
	where $\mu = \widehat{\omega} + |y|$. Then,
	\begin{align*}
		\|R(iy, A)\| & = \|(\mu - iy)R(iy, A_0)R(\mu, A) + R(\mu, A)\| \notag \\
		& \leq M \frac{2|y| + \widehat{\omega}}{|y|^\beta \cdot |y|^{\frac{1}{p}}} + \frac{M}{|y|^{\frac{1}{p}}} \notag \\
		& \leq \frac{\widetilde{M}}{|y|^\alpha},
	\end{align*}
	where $\alpha = \min\{\beta + 1/p -1, ~1/p\} > 0$. Let $K_2 > K_1$ be sufficiently large such that
	\[
	\|R(iy, A)\| \leq \frac{\|L\|^{-1}}{2},
	\]
	provided $|y| > K_2$. By \eqref{equ:gj1}, we have
	\[
	\| R(iy, (A+L)_0) \| \leq 2\|R(iy, A_0)\|.
	\]
	The proof is complete.
\end{proof}

Finally, we consider the analyticity.
\begin{thm}\label{thm:alm}
	Suppose that $A$ is an MR operator and $T_{A_0}$ is analytic. Then, for any $L \in \mathcal{L}(X_0,X)$, $T_{(A+L)_0}$ is analytic.
\end{thm}
\begin{proof}
	Assume $L \neq 0$. Since $T_{A_0}$ is analytic and $A$ is an MR operator, by \cite[Corollary 3.7.18]{ABHN11} and \autoref{lem:MR} (b), there is a constant $K > 0$, such that
	\[
	\|R(iy, A_0)\| \leq \frac{M}{|y|}, ~~\|R(\mu, A)\| \leq \varepsilon,
	\]
	provided $|y| > K$, where $\mu = |y|$, $\varepsilon = \frac{\|L\|^{-1}}{2(2M+1)}$. Hence,
	\begin{align*}
		\|R(iy, A)\| & = \|(\mu - iy)R(iy, A_0)R(\mu, A) + R(\mu, A)\| \\
		& \leq \frac{2|y|M}{|y|}\varepsilon + \varepsilon   = \|L\|^{-1}/2.
	\end{align*}
	The result follows from \eqref{equ:gj1} and \cite[Corollary 3.7.18]{ABHN11}.
\end{proof}

Combining with \autoref{lem:almost} (b) and \autoref{thm:alm}, we obtain the following perturbation of almost sectorial operators. See also \cite{DMP10} for more results on the relatively bounded perturbations of almost sectorial operators.
\begin{cor}
	If $A$ is a $p$-almost sectorial operator ($p \geq 1$), so is $A+L$ for any $L \in \mathcal{L}(X_0, X)$.
\end{cor}

\section{Critical and essential growth bound of perturbed semigroups}\label{criess}

We still assume $A$ is an MR operator, $L \in \mathcal{L}(X_0, X)$. In this section, we study the critical and essential growth bound of a perturbed semigroup $T_{(A+L)_0}$. In many applications, one would like to hope the following hold:
\[
\omega_{\mathrm{crit}}(T_{(A+L)_0}) \leq \omega_{\mathrm{crit}}(T_{A_0}), \quad \omega_{\mathrm{ess}}(T_{(A+L)_0}) \leq \omega_{\mathrm{ess}}(T_{A_0}),
\]
see \autoref{comments} for a short discussion. The following elementary lemma seems well known, especially for the case $j = 0$, see also the proof of \cite[Proposition 3.7]{BNP00}.
\begin{lem}\label{lem:exponent}
	Suppose $f_j: ~[0, \infty) \rightarrow [0, \infty)$, $ j = 0,1,2,\cdots $, satisfy
	\[
	f_j(t+s) \leq \sum_{k=0}^{j}f_k(t)f_{j-k}(s), ~\forall t,s > 0, ~j = 0,1,2,\cdots,
	\]
	and each $ f_j $ is bounded on any compact intervals.
	Set ($ \ln 0 := - \infty $)
	\[\label{asy}\tag{$ \divideontimes $}
	\omega := \lim_{t \rightarrow \infty} \frac{1}{t}\ln f_0(t) = \inf_{t>0}\frac{1}{t}\ln f_0(t) = \inf\{ \omega \in \mathbb{R}: \exists M \geq 0, f_0(t) \leq Me^{\omega t}, \forall t \geq 0 \}.
	\]
	Then, for any $\gamma > \omega, ~\forall j \geq 0, ~\exists M_j \geq 0$, such that $f_j(t) \leq M_je^{\gamma t}, ~\forall t \geq 0$.
\end{lem}
\begin{proof}
	Note that by a standard result (see, e.g., \cite[Lemma IV.2.3]{EN00}), the assumption on $ f_0 $ gives that \eqref{asy} holds.
	By the definition of $\omega$, for any $\gamma > \omega$, there is $t_0 > 0$, such that $f_0(t_0) \leq \frac{1}{2}e^{\gamma t_0}$. By induction, one gets
	\begin{align}
		f_j(t) & = f_j(t-t_0+t_0) \leq \sum_{k=0}^{j}f_k(t-t_0)f_{j-k}(t_0) \notag\\
		& \leq \frac{1}{2}e^{\gamma t_0}f_j(t-t_0) + M'_je^{\gamma t}, \label{eq:12aa}
	\end{align}
	where $t \geq t_0$, and the constant $M'_j$ doesn't depend on $t$.
	Set
	\[
	a_n(s) := f_j(nt_0 + s),~~s \in [0,t_0).
	\]
	Using \eqref{eq:12aa}, we have
	\begin{align*}
		a_n(s) & \leq \frac{1}{2}e^{\gamma t_0}a_{n-1}(s) + M'_je^{\gamma (nt_0 + s)} \\
		& \leq \frac{1}{2^2}e^{2 \gamma t_0}a_{n-2}(s) + (\frac{1}{2} + 1)M'_je^{\gamma (nt_0 + s)} \\
		& \leq \cdots \\
		& \leq \frac{1}{2^n}e^{n \gamma t_0}a_0(s) + (\frac{1}{2^{n-1}} + \frac{1}{2^{n-2}} + \cdots + 1)M'_je^{\gamma (nt_0 + s)} \\
		& \leq M_j e^{\gamma (nt_0 + s)},
	\end{align*}
	for some constant $M_j \geq 0$. This completes the proof.
\end{proof}

Since the definition of critical spectrum depends on the space $l^{\infty}_{T}(X)$, we need some preliminaries. See \cite{BNP00} for similar results. Recall the meaning of $S_A \diamond V$, $S_n$, $R_n$ (see \autoref{nddo}), $\widetilde{S}_n, \widehat{S}_n, \overline{S}_{n}$ (which are defined similarly as for a $ C_0 $ semigroup $ T $ in \autoref{spec}) and so on in \autoref{pre}.
\begin{lem}\label{lem:eq}
	$l_{T_{(A+L)_0}}^\infty(X_0) = l_{T_{A_0}}^\infty(X_0)$.
\end{lem}
\begin{proof}
	This follows from
	\[
	\| T_{(A+L)_0}(h) - T_{A_0}(h) \| = \| (S_A \diamond LT_{(A+L)_0})(h) \| \leq M\delta(h) \rightarrow 0, ~~\text{as}~ h \rightarrow 0^+.
	\]
\end{proof}
\begin{lem}\label{lem:abc}
	$\widetilde{S}_n(t) l_{T_{A_0}}^\infty(X_0) \subset l_{T_{A_0}}^\infty(X_0),~\forall t \geq 0$.
\end{lem}
\begin{proof}
	Consider
	\begin{align*}
		& T_{A_0}(h)S_n(t) - S_n(t) \\
		= & T_{A_0}(h)(S_A \diamond LS_{n-1})(t) - (S_A \diamond LS_{n-1})(t) \\
		= & (S_A \diamond LS_{n-1})(t+h) - (S_A \diamond LS_{n-1})(t) - (S_A \diamond LS_{n-1}(t+\cdot))(h) \\
		= & (S_A \diamond LS_{n-1}(h+\cdot))(t) - (S_A \diamond LS_{n-1})(t) + T_{A_0}(t)(S_A \diamond LS_{n-1})(h) + W_1(h) \\
		= & (S_A \diamond L\sum_{k=0}^{n-1}S_k(\cdot)S_{n-1-k}(h))(t) - (S_A \diamond LS_{n-1})(t) + W_2(h) \\
		= & (S_A \diamond LS_{n-1}(\cdot)[T_{A_0}(h) - I])(t) + W_3(h) \\
		= & S_n(t)(T_{A_0}(h) - I) + W_3(h),
	\end{align*}
	the second and third equality being a consequence of \eqref{equ:sf2}, and the fourth one being a consequence of \autoref{lem:srelation}, where
	\begin{align*}
		W_1(h) & := -(S_A \diamond LS_{n-1}(t+\cdot))(h), \\
		W_2(h) & := W_1(h) + T_{A_0}(t)(S_A \diamond LS_{n-1})(h), \\
		W_3(h) & := W_2(h) + \sum_{k=0}^{n-2}(S_A \diamond LS_k(\cdot)S_{n-1-k}(h))(t), \\
		& =  W_2(h) + \sum_{k=0}^{n-2}S_{k+1}(t)S_{n-1-k}(h).
	\end{align*}
	Note that for $k \geq 1$, $\|S_k(h)\| \leq \widetilde{M_k} \delta(h)$, for some constant $\widetilde{M_k}>0$. Hence, $\|W_3(h)\| \leq \widetilde{M}\delta(h)$, which shows the result.
\end{proof}
\begin{lem} \label{lem:sv}
	If $(S_A \diamond V) (t)$ is (right) norm continuous at $t_0$, then $(S_A \diamond V)(t_0)l^\infty(X_0) \subset l_{T_{A_0}}^\infty(X_0)$, where $V \in C_s([0,\infty), \mathcal{L}(X_0, X))$.
\end{lem}
\begin{proof}
	We need to show $\|T_{A_0}(h)(S_A \diamond V)(t_0) - (S_A \diamond V)(t_0)\| \rightarrow 0$, as $h \rightarrow 0^+$, which follows from
	\[
	T_{A_0}(h)(S_A \diamond V)(t_0) - (S_A \diamond V)(t_0) = (S_A \diamond V)(t_0 + h) - (S_A \diamond V)(t_0) - (S_A \diamond V(t_0+\cdot))(h).
	\]
\end{proof}

The following are main results.
\begin{lem}\label{lem:sexponent}
	\begin{enumerate}[(a)]
		\item In $\mathcal{L}(X_0) / \mathcal{K}(X_0)$, $\forall \gamma > \omega_{\mathrm{ess}}(T_{A_0}), ~\forall j \geq 0$, there is $M_j^1 \geq 0$, such that $\|\overline{S}_j(t)\| \leq M_j^1e^{\gamma t}$.
		\item In $l^\infty(X_0)/l_{T_{A_0}}^\infty(X_0)$, $\forall \gamma > \omega_{\mathrm{crit}}(T_{A_0}), ~\forall j \geq 0$, there is $ M_j^2 \geq 0$, such that $\|\widehat{S}_j(t)\| \leq M_j^2e^{\gamma t}$.
	\end{enumerate}
\end{lem}
\begin{proof}
	By \autoref{lem:abc}, $\widehat{S}_j$ can be defined in $l^\infty(X_0)/l_{T_{A_0}}^\infty(X_0)$. Thus, by \autoref{lem:srelation}, we have
	\[
	\widehat{S}_j(t+s) = \sum\limits_{k=0}^{j}\widehat{S}_k(t)\widehat{S}_{j-k}(s) ~\text{in}~ l^\infty(X_0)/l_{T_{A_0}}^\infty(X_0),
	\]
	and
	\[
	\overline{S}_j(t+s) = \sum\limits_{k=0}^{j}\overline{S}_k(t)\overline{S}_{j-k}(s) ~\text{in}~ \mathcal{L}(X_0) / \mathcal{K}(X_0).
	\]
	The proof is complete by \autoref{lem:exponent}.
\end{proof}
\begin{thm}\label{thm:main1}Let positive sequence $\{t_n\}$ satisfy $t_n \rightarrow \infty$, as $n \rightarrow \infty$.
	\begin{enumerate}[(a)]
		\item If there is $k \in \mathbb{N}$ such that $R_k(t_n)$ is compact, $n = 1, 2, 3, \cdots$, then $\omega_{\mathrm{ess}}(T_{(A+L)_0}) \leq \omega_{\mathrm{ess}}(T_{A_0})$.
		\item If there is $k \in \mathbb{N}$ such that $R_k$ is norm continuous at $t_1, t_2, t_3, \cdots$, then $\omega_{\mathrm{crit}}(T_{(A+L)_0}) \leq \omega_{\mathrm{crit}}(T_{A_0})$.
		\item If $R_1$ is compact (resp. norm continuous) at $t_1, t_2, t_3, \cdots$, then $\omega_{\mathrm{ess}}(T_{(A+L)_0}) = \omega_{\mathrm{ess}}(T_{A_0})$ (resp. $\omega_{\mathrm{crit}}(T_{(A+L)_0}) = \omega_{\mathrm{crit}}(T_{A_0})$).
	\end{enumerate}
\end{thm}
\begin{proof}
	(a) Since $R_k(t_n)$ is compact, in $\mathcal{L}(X_0) / \mathcal{K}(X_0)$, $\overline{R}_k(t_n) = 0$, for $n=1,2,3, \cdots$. Hence,
	\[
	\overline{T}_{(A+L)_0}(t_n) = \sum_{j=0}^{k-1}\overline{S}_j(t_n),
	\]
	for $n=1,2,3, \cdots$. By \autoref{lem:sexponent}, we have
	\[
	\omega_{\mathrm{ess}}(T_{(A+L)_0}) = \lim_{n \rightarrow \infty} \frac{1}{t_n}\ln \|\overline{T}_{(A+L)_0}(t_n)\| \leq \omega_{\mathrm{ess}}(T_{A_0}),
	\]
	which completes the proof of (a).

	(b) Note that due to \autoref{lem:eq}, $l^\infty(X_0)/l_{T_{(A+L)_0}}^\infty(X_0) = l^\infty(X_0)/l_{T_{A_0}}^\infty(X_0)$. Thus, by \autoref{lem:sv}, $\widehat{R}_k(t_n) = 0$, for $n=1,2,3,\cdots$. The following proof is similar as (a).

	(c) In this case $\overline{T}_{(A+L)_0}(t_n) = \overline{T}_{A_0}(t_n)$ (resp. $\widehat{T}_{(A+L)_0}(t_n) = \widehat{T}_{A_0}(t_n)$), for $n=1,2,3, \cdots$, which shows the result.
\end{proof}

In many cases, since we don't know the explicit expression of $T_{(A+L)_0}$, calculating $R_k$ is hard. In the following, we will give more elaborate result than \autoref{thm:main1}. Consider $T_{A_0}$ as the perturbation of $T_{(A+L)_0}$ under $-L$. Set
\[
\mathcal{D}(W) := -S_{A+L} \diamond LW,~W \in C_s([0,\infty), \mathcal{L}(X_0, X)).
\]
\begin{lem}\label{lem:permm}
	The following statements are equivalent, where $P$ stands for norm continuity, or compactness on $(0,\infty)$.
	\begin{enumerate}[(a)]
		\item $\mathcal{B}^k(T_{(A+L)_0}) = R_k$ has property $P$ ($\Leftrightarrow$ $\mathcal{B}^k(T_{A_0}) = S_k$ has property $P$).
		\item $\mathcal{D}^k(T_{(A+L)_0})$ has property $P$ ($\Leftrightarrow$ $\mathcal{D}^k(T_{A_0})$ has property $P$).
	\end{enumerate}
\end{lem}
\begin{proof}
	Here note that if $ S_{n} $ (or $ R_{n} $) is compact on $ (0, \infty) $, then it is norm continuous on $ (0, \infty) $ by \autoref{cor:DP} \eqref{DPcc} (or \eqref{DPdd}). So if $ P $ stands for compactness on $(0,\infty)$, then it stands for norm continuity and compactness on $(0,\infty)$.
	It suffices to consider one direction, e.g., (a) $\Rightarrow$ (b). Since
	\[
	S_{A+L} = S_A + S_A \diamond LS_{A + L},
	\]
	we have
	\[
	\mathcal{D} = -\mathcal{B} + \mathcal{B}\mathcal{D}.
	\]
	We prove that $\mathcal{B}^{k-l}\mathcal{D}^l(T_{(A+L)_0})$ has property $P$ by induction, where $l=0,1,2,\cdots,k$. The case $l=0$ is (a). Since
	\[
	\mathcal{B}^{k-l}\mathcal{D}^l = \mathcal{B}^{k-l}\mathcal{D}^{l-1}\mathcal{D} = -\mathcal{B}^{k-(l-1)}\mathcal{D}^{l-1} + \mathcal{B}\mathcal{B}^{k-l}\mathcal{D}^l,
	\]
	and by induction, $\mathcal{B}^{k-(l-1)}\mathcal{D}^{l-1}(T_{(A+L)_0})$ has property $P$, it yields $\mathcal{B}^{k-l}\mathcal{D}^l(T_{(A+L)_0})$ has property $P$ by \autoref{lem:regfix}. Note that the equivalences in the brackets are \autoref{corr:sr}. This completes the proof.
\end{proof}
\begin{thm}\label{thm:main2}
	\begin{enumerate}[(a)]
		\item If there is $S_k$ (or $R_k$) such that it is compact on $(0,\infty)$, then $\omega_{\mathrm{ess}}(T_{(A+L)_0}) = \omega_{\mathrm{ess}}(T_{A_0})$.
		\item If there is $S_k$ (or $R_k$) such that it is norm continuous on $(0,\infty)$, then $\omega_{\mathrm{crit}}(T_{(A+L)_0}) = \omega_{\mathrm{crit}}(T_{A_0})$.
	\end{enumerate}
\end{thm}
\begin{proof}
	Use \autoref{lem:permm} and apply \autoref{thm:main1} twice.
\end{proof}

Using \autoref{corr:s1} and \autoref{thm:main2}, we have the following corollary.
\begin{cor}\label{cor:lt}Let $A$ be an MR operator, $L \in \mathcal{L}(X_0, X)$.
	\begin{enumerate}[(a)]
		\item If $LT_{A_0}$ is norm continuous on $(0,\infty)$, then $\omega_{\mathrm{crit}}(T_{(A+L)_0}) = \omega_{\mathrm{crit}}(T_{A_0})$.
		\item If $LT_{A_0}$ is norm continuous and compact on $(0,\infty)$, then $\omega_{\mathrm{ess}}(T_{(A+L)_0}) = \omega_{\mathrm{ess}}(T_{A_0})$.
	\end{enumerate}
\end{cor}

Using \autoref{corr:s2} and \autoref{thm:main2}, we have the following corollary due to \cite{DLM08}.
\begin{cor}[{\cite[Theorem 1.2]{DLM08}}]\label{cor:lmr}
	Suppose that $A$ is a quasi Hille--Yosida operator and $LT_{A_0}$ is compact on $(0,\infty)$, then $\omega_{\mathrm{ess}}(T_{(A+L)_0}) = \omega_{\mathrm{ess}}(T_{A_0})$.
\end{cor}
\begin{rmk}
	It was only shown in \cite[Theorem 1.2]{DLM08} that $\omega_{\mathrm{ess}}(T_{(A+L)_0}) \leq \omega_{\mathrm{ess}}(T_{A_0})$. However, using \cite[Theorem 1.2]{DLM08}, it also yields  $\omega_{\mathrm{ess}}(T_{(A+L)_0}) = \omega_{\mathrm{ess}}(T_{A_0})$, because $LT_{(A+L)_0}$ is also compact on $(0,\infty)$, see the proof in \autoref{corr:s2}.
\end{rmk}

\section{Age-structured population models in $ L^p $ spaces}\label{model}

Consider the following age-structured population model in $ L^p $ (see \cite{Web85, Web08, Thi91, Thi98, Rha98, BHM05, MR09a}):
\begin{equation}\label{equ:age}
\begin{cases}
(\partial_t + \partial_a)u(t, a) = A(a)u(t, a), ~t > 0, a \in (0, c),\\
u(t, 0) = \int_{0}^{c} C(a) u(t, a) ~\mathrm{d} a, ~ t > 0,\\
u(0, a) = u_0(a), ~ a \in (0, c), ~ u_0 \in L^p((0,c), E),
\end{cases}
\end{equation}
where $ 1 \leq p < \infty $, $ c \in (0, \infty] $ and $ E $ is a Banach space ($ u(t, a) \in E $); for all $ a \in (0,c) $, $ C(a): E \to E $ are bounded linear operators and $ A(a) : E \to E $ are closed linear operators (the detailed assumptions are given in \autoref{sub:ass}).
The approach given in this section to the study of model \eqref{equ:age} goes back to Thieme \cite{Thi91} as early as 1991. It seems that this model in $ L^p((0, c), E) $ ($ p > 1 $) was first investigated by Magal and Ruan in \cite{MR07} (see also \cite{Thi08}). We notice that, in control theory (and approximation theory), age-structured population models can be considered as boundary control systems and in this case the state space is often taken as $ L^p((0, c), E) $ ($ p > 1 $); see, e.g., \cite{CZ95}. In addition, when $ c = \infty $, since $ L^{p}((0,\infty), E) $ does not include in $ L^{1}((0,\infty), E) $, this gives in some sense more solutions of \eqref{equ:age} by taking initial data $ u_0 \in \bigcup_{1 \leq p < \infty} L^{p}((0,\infty), E) $. Finally, we mention that the geometric property of $ L^2((0, c), E) $ is usually better than $ L^1((0, c), E) $.

\subsection{evolution family: mostly review}
Before giving some naturally standard assumptions on \eqref{equ:age}, we first recall some backgrounds about evolution family.
We say $ \{U(a, s)\}_{0 \leq s \leq a < c} $ is an exponentially bounded linear evolution family (or for short \emph{evolution family}) if it satisfies the following:
\begin{enumerate}[(1)]
	\item $ U(a, s) \in \mathcal{L}(E, E) $, $ U(a, s) = U(a, \tau) U(\tau, s) $ and $ U(a, a) = I $ for all $ 0 \leq s \leq \tau \leq a < c $;
	\item $ (a, s, x) \mapsto U(a, s)x $ is continuous in $ 0 \leq s \leq a < c $ and $ x \in E $;
	\item there exist constants $ C \geq 1 $ and $ \omega \in \mathbb{R} $ such that
	\[
	|U(a,s)| \leq C e^{\omega (a - s)}, ~ 0 \leq s \leq a < c.
	\]
\end{enumerate}

Define the Howland semigroup $ T_0 $ on $ L^p((0, c), E) $ (with respect to $ \{U(a, s)\}_{0 \leq s \leq a < c} $) as
\begin{equation}\label{equ:how}
(T_0(t) \varphi) (a) = \begin{cases}
U(a, a -t)\varphi(a - t), &~ a \geq t,\\
0, &~ a \in [0, t],
\end{cases}
\end{equation}
which is a $ C_0 $ semigroup, and let $ \mathcal{B}_0 $ denote its generator.
Note that $ D(\mathcal{B}_0) \subset C([0,c), E) $. In fact, a simple computation shows that if $ \lambda \in \rho(\mathcal{B}_0) $, then
\[
(R(\lambda, \mathcal{B}_0) f )(a) = \int_{0}^{a} e^{\lambda(a - s)} U(a, s) f(s) ~\mathrm{d} s, ~ f \in L^p((0, c), E).
\]
\begin{proof}
	Let $ R(\lambda, \mathcal{B}_0) f = \varphi $.
	Note that we have
	\[
	e^{-\lambda t}T_0(t) \varphi - \varphi = (\lambda - \mathcal{B}_0) \int_{0}^{t} e^{-\lambda s} T_0(s) \varphi ~\mathrm{d} s = \int_{0}^{t} e^{-\lambda s} T_0(s) f ~\mathrm{d} s.
	\]
	For $ t \geq a $, we see $ (e^{-\lambda t}T_0(t) \varphi - \varphi) (a) = - \varphi (a) $ and
	\[
	\int_{0}^{t} (e^{-\lambda s} T_0(s) f) (a) ~\mathrm{d} s = \int_{0}^{a} e^{-\lambda s} U(a, a - s) f(a - s) ~\mathrm{d} s = - \int_{0}^{a} e^{\lambda(a - s)} U(a, s) f(s) ~\mathrm{d} s,
	\]
	completing the proof.
\end{proof}

\begin{thm}[See {\cite[Section 3.3]{CL99}}]\label{thm:evo}
	$ \omega(T_0) = s(\mathcal{B}_0) = \omega(U) $, where $ \omega(U) $ is the growth bound of $ \{U(a,s)\} $, i.e.,
	\[
	\omega(U) : = \inf \{ \omega: ~\text{there is $ M(\omega) \geq 1 $ such that}~ |U(a,s)| \leq M(\omega)e^{\omega(a - s)} ~\text{for all}~ a \geq s \}.
	\]
\end{thm}

\begin{rmk}
	The sufficient and necessary conditions such that $ \mathcal{B}_0 $ is a generator of Howland semigroup \eqref{equ:how} induced by an evolution family were given by Schnaubelt \cite[Theorem 2.6]{Sch96} (see also \cite[Theorem 2.4]{RSRV00}).
\end{rmk}

Let us use the evolution family $ \{ U(a,s) \}_{0 \leq s \leq a < c} $ to study the sum $ -\frac{\mathrm{d}}{\mathrm{d} a} + A(\cdot) $ on $ L^p((0, c), E) $ with $ D(-\frac{\mathrm{d}}{\mathrm{d} a} + A(\cdot)) = D(\frac{\mathrm{d}}{\mathrm{d} a}) \cap D(A(\cdot)) $, where
\[
\frac{\mathrm{d}}{\mathrm{d} a} : f \mapsto \frac{\mathrm{d} f}{\mathrm{d} a}, ~D(\frac{\mathrm{d}}{\mathrm{d} a}) = \{ f \in W^{1,p}((0, c), E): f(0) = 0 \},
\]
and $ A(\cdot) $ is the multiplication operator, i.e., $ (A(\cdot) f(\cdot))(a) = A(a) f(a) $, $ a \in (0,c) $ with
\[
D(A(\cdot)) = \{ f \in L^p((0, c), E): f(a) \in D(A(a)) ~\text{a.e.}~ a \in (0,c), A(\cdot) f(\cdot) \in L^p((0, c), E) \}.
\]

\begin{defi}\label{def:evo}
	We say $ \{A(a)\}_{0 \leq a < c} $ generates an exponentially bounded linear evolution family $ \{U(a, s)\}_{0 \leq s \leq a < c} $ (in $ L^p $-sense) if $ \mathcal{B}_0 $ is a closure of $ -\frac{\mathrm{d}}{\mathrm{d} a} + A(\cdot) $ (with $ D(-\frac{\mathrm{d}}{\mathrm{d} a} + A(\cdot)) = D(\frac{\mathrm{d}}{\mathrm{d} a}) \cap D(A(\cdot)) $) on $ L^p((0, c), E) $.
\end{defi}

\begin{rmk}
	Consider the following abstract non-autonomous (linear) Cauchy problems:
	\[\tag{ACP} \label{acp}
	\dot{x}(a) = A(a) x(a), ~x(s) = x_s \in D(A(s)), ~ a \geq s,
	\]
	where $ \overline{D(A(s))} = X $ for all $ s \in [0,c) $. An evolution family $ \{U(a, s)\}_{0 \leq s \leq a < c} $ is said to \emph{solve the above Cauchy problem} \eqref{acp} if for each $ s \in [0,c) $, there is a dense linear space $ Y_{s} \subset D(A(s)) $ such that for each $ x_s \in Y_{s} $, $ x(a) = U(a, s) x_s $ ($ s \leq a < c $) is $ C^1 $ and satisfies \eqref{acp} point-wisely; see, e.g., \cite{Sch96} (or \cite[Definition VI.9.1]{EN00}). Now if $ \{U(a, s)\}_{0 \leq s \leq a < c} $ solves \eqref{acp}, then $ \{A(a)\}_{0 \leq a < c} $ generates $ \{U(a, s)\}_{0 \leq s \leq a < c} $ (in $ L^p $-sense); see \cite{Sch96} (or the proof of \cite[Theorem 3.12]{CL99}).
\end{rmk}

Consider the following examples.
\begin{exa}
	\begin{enumerate}[(a)]
		\item Assume $ A_0 $ is a generator of a $ C_0 $ semigroup $ T_0 $ and $ b(\cdot) \in C(\mathbb{R}, \mathbb{R}_+) $ (such that $ \inf b(\cdot) > 0 $). Let
		\[
		A(a) = b(a) A_0, ~ 0 \leq a < c.
		\]
		Then, for $ U(a,s) = T_0(\int_{s}^{a} b(r) ~\mathrm{d}r) $, $ \{U(a, s)\}_{0 \leq s \leq a < c} $ solves \eqref{acp}.
		\item Assume $ A_0 $ is a generator of a $ C_0 $ semigroup and $ L: (0,c) \mapsto L(E, E) $ is strongly continuous and $ \sup_{t \in [0,c)} |L| < \infty $. Let
		\[
		A(a) = A_0 + L(a), ~ 0 \leq a < c.
		\]
		Then, it is easy to see $ \{A(a)\}_{0 \leq a < c} $ generates an evolution family $ \{U(a, s)\}_{0 \leq s \leq a < c} $; see, e.g., the proof of \cite[Proposition 6.23]{CL99}.
	\end{enumerate}
\end{exa}
For the above examples, using the classical solutions of \eqref{acp} to give evolution family sometimes is limited; other way by using the mild solutions in the sense of \cite[Definition 3.1.1]{ABHN11} can be found in \cite[Section 3]{Che18c}. See also \cite[Section 5.6--5.7]{Paz83} for the ``parabolic'' type (in the sense of Tanabe) and \cite[Section 5.3--5.5]{Paz83} for the ``hyperbolic'' type (in sense of Kato) which in some contexts $ \{A(a)\}_{0 \leq a < c} $ can generate evolution family (see also \cite{Sch02} for a survey).

\subsection{standard assumptions and main results}\label{sub:ass}

Hereafter, we make the following assumptions.

\begin{enumerate}[\bfseries (H1)]
	\item (about $ \{A(a)\} $) Assume $ \{A(a)\}_{0 \leq a < c} $ generates an exponentially bounded linear evolution family $ \{U(a, s)\}_{0 \leq s \leq a < c} $; see \autoref{def:evo}. Let $ T_0 $ be its corresponding Howland semigroup (see \eqref{equ:how}) with the generator $ \mathcal{B}_0 $.
	\item (about $ \{C(a)\} $) Assume $ C(\cdot): (0, c) \to \mathcal{L}(E, E) $ is strongly continuous. In addition, there is a positive function $ \gamma \in L^{p'}(0,c) $ such that $ |C(a)| \leq \gamma(a) $ where $ 1/{p'} + 1/p = 1 $.
\end{enumerate}

Set $ X = E \times L^p((0, c), E) $, $ X_0 = \{0\} \times L^p((0, c), E) $, and
\begin{equation}\label{equ:boundary}
\mathcal{L}: X_0 \to X, (0,\varphi) \mapsto (L\varphi,0), ~ L\varphi = \int_{0}^{c} C(a)\varphi(a) ~\mathrm{d} a.
\end{equation}
Moreover, there is a unique closed operator $ \mathcal{A} $ with $ \overline{D(\mathcal{A})} = X_0 $ such that for all $ (y, f) \in X $ and $ \lambda \in \rho(\mathcal{B}_0) $, $ R(\lambda, \mathcal{A}) (y, f) = (0, \varphi) $, where
\[
\varphi(a) = e^{-\lambda a} U(a, 0) y + (R(\lambda, \mathcal{B}_0)f)(a),~ a \in (0, c),
\]
see \cite[Lemma 6.2]{MR07}; \cite{Thi91} contains more descriptions of $ \mathcal{A} $.
Note that by the assumption \textbf{(H2)} (on $ \{C(a)\} $), we see $ L $ and so $ \mathcal{L} $ are bounded.
Now, the solutions of the age-structured population model \eqref{equ:age} can be interpreted as the mild solutions of the following abstract Cauchy problem:
\[\label{equ:CP}\tag{CP}
\begin{cases}
\dot{U}(t) = (\mathcal{A} + \mathcal{L}) U(t),\\
U(0) = U_0 \in X_0.
\end{cases}
\]
Recall that $ U \in C([0,\tau], X) $ is called a \emph{mild solution} of \eqref{equ:CP} if $ \int_{0}^{s} U(r) ~\mathrm{d} r \in D(\mathcal{A} + \mathcal{L}) $ and
\[
U(s) = U(0) + (\mathcal{A} + \mathcal{L}) \int_{0}^{s} U(r) ~\mathrm{d} r, ~s \in [0,\tau],
\]
see, e.g., \cite[Definition 3.1.1]{ABHN11}.
\begin{defi}\label{defi:age}
	If $ U = (0, u) $ is a mild solution of \eqref{equ:CP}, then $ u $ is called a mild solution of the age-structured population model \eqref{equ:age}.
\end{defi}
Note that if $ u \in W^{1, p}([0, \tau] \times [0, c), E) $ satisfies \eqref{equ:age} a.e. (or $ u \in C^1([0, \tau] \times [0, c], E) $ with $ c < \infty $ satisfies \eqref{equ:age} point-wisely), then $ u $ is a mild solution of \eqref{equ:age}. See \autoref{lem:presentation} for a characterization of mild solutions of the model \eqref{equ:age} in terms of themselves.

By a simple computation, we see $ \rho(\mathcal{A}) = \rho(\mathcal{B}_0) $, $ \mathcal{A}_0: = \mathcal{A}_{X_0} = 0 \times \mathcal{B}_0 $ and for $ \lambda \in \rho(\mathcal{A}) $,
\begin{equation}\label{equ:re}
R(\lambda, \mathcal{A})^n (y, f) = (0,\varphi_n), ~\varphi_n (a) = \frac{1}{(n -1)!} a^{n-1} e^{-\lambda a} U(a, 0) y + (R(\lambda, \mathcal{B}_0)^n f) (a).
\end{equation}
Particularly, we obtain
\begin{lem}
	If $ p = 1 $, then $ \mathcal{A} $ is a Hille--Yosida operator and so is $ \mathcal{A} + \mathcal{L} $.
\end{lem}
Consider the case $ p > 1 $. Since $ \mathcal{A}_0 $ ($ \cong \mathcal{B}_0 $) generates a $ C_0 $ semigroup $ (0, T_0) $, we know $ \mathcal{A} $ generates an integrated semigroup $ S_{\mathcal{A}} $ defined by (see \eqref{equ:biaodashi})
\[
S_{\mathcal{A}} (t) (y, f) = (0, U_0(t)y + \int_{0}^{t}T_0(s) f ~\mathrm{d} s),
\]
where
\[
(U_0(t)y)(a) = \begin{cases}
U(a, 0)y, & a \leq t,\\
0, & a > t.
\end{cases}
\]

For $ f_1 \in L^p((0,t_1), E) $ and $ f_2 \in L^p((0,t_1), L^p((0,c), E)) $ ($ \cong L^p((0,t_1) \times (0,c), E) $), we have
\begin{equation}\label{equ:sa}
(S_{\mathcal{A}} \diamond (f_1, f_2))(t) (a) =
\begin{cases}
(0, U(a, 0)f_1(t-a) + (T_0 * f_2)(t)(a)),  & a \leq t, \\
(0, (T_0 * f_2)(t)(a)), & a > t.
\end{cases}
\end{equation}
So, it's clear to see that there exists $ \widehat{M} \geq 1 $ independent of $ (f_1, f_2) $ and $ t_1 $ such that
\[
(\int_{0}^{c} |(S_{\mathcal{A}} \diamond (f_1, f_2))(t) (a)|^p ~\mathrm{d} a)^{1/p} \leq \widehat{M} \left\|e^{\omega(t_1-\cdot)}(f_1, f_2)(\cdot) \right\|_{L^p(0,t_1;X)}, ~\forall t \in [0,t_1].
\]
That is, we have the following; see also the proof of \cite[Theorem 6.6]{MR07} and the following of Proposition 5.6 in \cite{Thi08}.
\begin{lem}
	If $ p > 1 $, then $ \mathcal{A} $ is a $ p $-quasi Hille--Yosida operator and so is $ \mathcal{A} + \mathcal{L} $.
\end{lem}

As $ \mathcal{A} $ is a quasi Hille--Yosida operator and thus the mild solutions of the linear equation \eqref{equ:CP} are given by
\[
U(t) = T_{\mathcal{A}_0} (t) U(0) + (S_{\mathcal{A}} \diamond \mathcal{L} U(\cdot))(t), ~U(0) \in X_0,
\]
returning to the second component of $ U = (0, u) $, by using $ T_{\mathcal{A}_0} = (0, T_{0}) $ and \eqref{equ:sa}, we obtain the following representation of the mild solutions of \eqref{equ:age} defined in \autoref{defi:age}; see also \cite[Section IV]{Thi91}.

\begin{lem}\label{lem:presentation}
	$ u $ is a mild solution of \eqref{equ:age} with $ u(0, \cdot) = u_0 \in L^p((0, c), E) $ if and only if $ u $ satisfies
	\[
	u(t, a) = \begin{cases}
	U(a, a - t)u_0(a - t), & a \geq t, \\
	U(a, 0) \int_{0}^{c} C(s) u(t - a, s) ~\mathrm{d}s, & a \in [0, t].
	\end{cases}
	\]
\end{lem}

\begin{rmk}\label{rmk:noHY}
	If there are a function $ h: (0, \infty) \mapsto \mathbb{R}_+ $ and $ y \in E $ such that $ |U(a, 0)y| \geq h(a) $ and
	\[
	\limsup_{\lambda \to \infty} \lambda^p \int_{0}^{\infty} e^{-\lambda p a} h^p(a) da > 0,
	\]
	then from \eqref{equ:re}, we know $ \mathcal{A} $ is not a Hille--Yosida operator when $ p > 1 $ and $ c = \infty $. For instance, $ U(a, 0) = I $ for all $ a \in (0,\infty) $.
\end{rmk}

Next, let us compute $ \mathcal{L}S_{\mathcal{A}}\diamond \mathcal{L}(0, T_0) = \mathcal{L}S_{\mathcal{A}}\diamond (LT_0, 0) $. Write $ \mathcal{L}S_{\mathcal{A}}\diamond (LT_0, 0) = (F, 0) $, then by \eqref{equ:sa}, we get
\[
F(t) = \begin{cases}
\int_{0}^{t} C(t-a) U(t-a, 0) LT_0(a) ~\mathrm{d}a, & t \leq c,\\
\int_{t-c}^{t} C(t-a) U(t-a, 0) LT_0(a) ~\mathrm{d}a, & t > c.
\end{cases}
\]

\begin{lem}\label{lem:pre}
	\begin{enumerate}[(a)]
		\item Suppose $ U(\cdot, 0) $ is norm continuous on $ (0, c) $, and $ C(\cdot) $ is norm measurable on $ (0, c) $ if $ p > 1 $ and is (essentially) norm continuous if $ p = 1 $. Then, $ F $ (i.e., $ \mathcal{L}S_{\mathcal{A}}\diamond (LT_0, 0) $) is norm continuous on $ (0, \infty) $.
		\item If $ U(\cdot, 0) $ is compact on $ (0, c) $, then $ F $ (i.e., $ \mathcal{L}S_{\mathcal{A}}\diamond (LT_0, 0) $) is norm continuous and compact on $ (0, \infty) $.
	\end{enumerate}
\end{lem}
\begin{proof}
	We only consider the case $ t < c $. Set
	\[
	F_{\epsilon}(t) = \int_{0}^{t-\epsilon} C(t-a) U(t-a, 0) LT_0(a) ~\mathrm{d}a.
	\]
	Note that by the condition on $ \{C(a)\} $ and the boundedness of $ \{U(t-a, 0) LT_0(a)\}_{0 \leq a \leq t} $, we see $ F_{\epsilon}(t) \to F(t) $ as $ \epsilon \to 0 $ in the uniform operator topology. So it suffices to consider $ F_{\epsilon}(t) $. Let $ 0 < |h| < \epsilon $ (with $ \epsilon $) be small and let $ p' $ satisfy $ 1/{p'} + 1/p = 1 $ ($ 1 < p' \leq \infty $). In the following, denote by $ \widetilde{C} > 0 $ the universal constant independent of $ h $ which might be different line by line.
	\begin{align*}
	& |F_{\epsilon}(t+h) - F_{\epsilon}(t)| \\
	= &~ |\int_{0}^{t + h -\epsilon} C(t + h - a) U(t + h - a, 0) LT_0(a) ~\mathrm{d}a - \int_{0}^{t -\epsilon} C(t - a) U(t - a, 0) LT_0(a) ~\mathrm{d}a| \\
	= & ~ |\int_{t - \epsilon}^{t + h -\epsilon} C(t + h - a) U(t + h - a, 0) LT_0(a) ~\mathrm{d}a \\
	& \quad + \int_{0}^{t -\epsilon} \{C(t + h - a) U(t + h - a, 0) - C(t - a) U(t - a, 0)\} LT_0(a) ~\mathrm{d}a |\\
	& \leq \widetilde{C} h + R_1 + R_2,
	\end{align*}
	where
	\begin{align*}
	R_1 & = \int_{0}^{t -\epsilon} |\{C(t + h - a) - C(t - a)\} U(t - a, 0) LT_0(a)| ~\mathrm{d}a, \\
	& \leq \left(\int_{0}^{t -\epsilon} |\{C(t + h - a) - C(t - a)\}U(t - a, 0)|^{p'} ~\mathrm{d}a\right)^{1/{p'}} \left(\int_{0}^{t -\epsilon} |LT_0(a)|^p ~\mathrm{d}a\right)^{1/p} \\
	& \leq \widetilde{C} \left(\int_{0}^{t -\epsilon} |\{C(t + h - a) - C(t - a)\}U(t - a, 0)|^{p'} ~\mathrm{d}a\right)^{1/{p'}},
	\end{align*}
	and
	\begin{align*}
	R_2 & = \int_{0}^{t -\epsilon} |C(t + h - a) \{ U(t + h - a, 0) - U(t - a, 0)\} LT_0(a)| ~\mathrm{d}a \\
	& \leq \left(\int_{0}^{t -\epsilon} |C(t + h - a)|^{p'} ~\mathrm{d} a\right)^{1/{p'}} \left(\int_{0}^{t -\epsilon} |\{ U(t + h - a, 0) - U(t - a, 0)\} LT_0(a)|^{p} ~\mathrm{d}a\right)^{1/p} \\
	& \leq \widetilde{C} \left(\int_{0}^{t -\epsilon} |U(t + h - a, 0) - U(t - a, 0)|^{p} ~\mathrm{d}a\right)^{1/p}.
	\end{align*}

	To prove (a), by the condition on $ \{U(a, 0)\} $, we know $ R_2 \to 0 $ as $ h \to 0 $, and by the condition on $ \{C(a)\} $, we get
	\[
	R_1 \leq \widetilde{C} \left(\int_{0}^{t -\epsilon} |C(t + h - a) - C(t - a)|^{p'} ~\mathrm{d}a\right)^{1/{p'}} \to 0, ~\text{as}~ \to 0,
	\]
	since $ C(\cdot): (0,c) \to \mathcal{L}(E, E) $ is Bochner $ p' $-integrable (in the uniform operator topology) if $ p' < \infty $ and is (essentially) norm continuous if $ p' = \infty $. This shows that $ F_{\epsilon} $ is norm continuous at $ t $.

	To prove (b), note first that if $ U(\cdot, 0) $ is compact on $ (0, c) $, then it is also norm continuous on $ (0, c) $. Indeed, if $ s > 0 $ and $ a > 0 $, then
	\[
	U(a+s, 0) - U(a,0) = (U(a+s, a) - I)U(a,0) \to 0,
	\]
	as $ s \to 0^+ $ in the uniform operator topology due to the compactness of $ U(a,0) $; the left norm continuity can be considered similarly. In particular, $ R_2 \to 0 $ as $ h \to 0 $. Since $ \{|C(a)|\}_{\epsilon \leq a \leq 2t} $ is bounded (due to the strong continuity of $ C(\cdot) $ on $ (0,c) $) and $ U(\cdot, 0) $ is compact on $ [\epsilon, t] $, by the Lebesgue dominated convergence theorem (see, e.g., \cite[Theorem 1.1.8]{ABHN11}), we get
	\[
	\lim_{h \to 0}R_1 \leq \widetilde{C} \left(\int_{0}^{t -\epsilon} \lim_{h \to 0} |\{C(t + h - a) - C(t - a)\}U(t - a, 0)|^{p'} ~\mathrm{d}a\right)^{1/{p'}} = 0.
	\]
	Thus, $ F_{\epsilon} $ is norm continuous at $ t $. The compactness of $ F_{\epsilon}(t) $ follows from \cite[Theorem C.7]{EN00} since $ U(\cdot, 0) $ is compact on $ [\epsilon, t] $. The proof is complete.
\end{proof}

Now we can give some asymptotic behaviors of the age-structured population model \eqref{equ:age}; see \autoref{thm:evo} for the characterization of $ \omega(T_0) $.

\begin{thm}\label{thm:age}
	Let one of the following conditions hold:
	\begin{enumerate}[(a)]
		\item $ U(\cdot, 0) $ is norm continuous on $ (0, c) $, and $ C(\cdot) $ is norm measurable on $ (0, c) $ if $ p > 1 $ and is (essentially) norm continuous if $ p = 1 $;
		\item $ U(\cdot, 0) $ is compact on $ (0, c) $;
		\item $ L $ (see \eqref{equ:boundary}) is compact.
	\end{enumerate}
	Let $ T_{(\mathcal{A}+\mathcal{L})_0} $ be the $ C_0 $ semigroup generated by the part of $ \mathcal{A}+\mathcal{L} $ in $ L^p((0,c), E) $. Then, the following statements hold.
	\begin{enumerate}[(1)]
		\item If $ c < \infty $ and condition (b) or (c) holds, then $ T_{(\mathcal{A}+\mathcal{L})_0} $ is eventually compact (and so eventually norm continuous). Particularly, $ \omega (T_{(\mathcal{A}+\mathcal{L})_0}) = s(T_{(\mathcal{A}+\mathcal{L})_0}) $ (see \eqref{equ:bound}).
		\item If condition (b) or (c) holds, then $ \omega_{\mathrm{ess}}(T_{(\mathcal{A}+\mathcal{L})_0}) = \omega_{\mathrm{ess}}(T_{0}) $. Particularly, if $ \omega_{\mathrm{ess}}(T_{0}) < 0 $, then $ T_{(\mathcal{A}+\mathcal{L})_0} $ is quasi-compact (see \cite[Section V.3]{EN00} for more consequences of this).
		\item $ \omega_{\mathrm{crit}}(T_{(\mathcal{A}+\mathcal{L})_0}) = \omega_{\mathrm{crit}}(T_{0}) $.
	\end{enumerate}
\end{thm}

\begin{proof}
	If condition (a) (resp. (b)) holds, then by \autoref{lem:pre} and \autoref{lem:NC} (b) (resp. \autoref{lem:CP} (b)), we know $ S_2 \triangleq S_{\mathcal{A}}\diamond \mathcal{L}S_{\mathcal{A}}\diamond (LT_0, 0) $ is norm continuous (resp. compact) on $ (0, \infty) $. If condition (c) holds (i.e., $ \mathcal{L} $ is compact), then by \autoref{corr:s2}, we have $ S_2 $ is compact on $ (0, \infty) $.

	To prove (1), as $ c < \infty $, we have for $ t > c $, $ T_0(t) = 0 $; so particularly $ T_0 $ is eventually compact. Conclusion (1) now follows \autoref{thm:perreg} (b).
	Finally, conclusions (2) and (3) are direct consequence of \autoref{thm:main2}.
\end{proof}

\begin{exa}\label{exa:model}
	In order to verify \autoref{thm:age}, consider the following examples which might be \emph{not} as general as possible.
	\begin{asparaenum}[\bfseries (a)]
		\item (See \cite[Section 5]{MR09a} and \cite{Web85}.) Let $ E = \mathbb{R}^n $. Then, $ C(a) $ and $ A(a) $ can be considered as matrices. Assume $ a \mapsto A(a) $ is continuous on $ [0,c) $ and $ a \mapsto C(a) $ is continuous on $ (0,c) $ with with $ \sup_{0<a<c}\{|A(a)|\} < \infty $ and $ |C(\cdot)| \in L^{p'}(0, c) $ ($ 1/{p'} + 1/p = 1 $), then assumptions \textbf{(H1) (H2)} hold. Note also that in this case $ L $ is compact. Now all the cases in \autoref{thm:age} are fulfilled. A very special case is $ A(a) \equiv - \mu I $ where $ \mu > 0 $; now $ U(t,s) = e^{-\mu (t - s)} I $ and so by \autoref{rmk:noHY} in this case $ \mathcal{A} $ is not a Hille--Yosida operator if $ p > 1 $ and $ c = \infty $; in addition, $ \omega_{\mathrm{ess}}(T_{(\mathcal{A}+\mathcal{L})_0}) < 0 $ as $ \omega(T_0) = \omega(U) \leq - \mu < 0 $.

		\item \label{exbb} (See \cite[Section 5]{BHM05}.)
		Consider the following model:
		\begin{equation*}
		\begin{cases}
		(\partial_t + \partial_a)u(t, a, s) = -\mu(a)u(t, a, s) + k(a) \Delta_s u(t, a, s), ~t > 0, a \in (0, c), s \in \mathbb{R}^n, \\
		u(t, 0, s) = \int_{0}^{c} \beta(a,s) u(t, a, s) ~\mathrm{d} a, ~ t > 0, s \in \mathbb{R}^n, \\
		u(0, a, s) = u_0(a, s), ~ a \in (0, c), s \in \mathbb{R}^n, ~ u_0 \in L^p((0,c), L^q(\mathbb{R}^n)).
		\end{cases}
		\end{equation*}
		Let $ E = L^q(\mathbb{R}^n) $ ($ 1 \leq q < \infty $). Take $ A(a) = -\mu (a) I + k(a) \Delta_s $ and $ C(a) = \beta (a, \cdot) $, where
		\begin{enumerate}[(i)]
			\item $ \Delta_s $ denotes the Laplace operator on $ L^q(\mathbb{R}^n) $ (i.e., $ \Delta f = \sum_{j = 1}^{n} D^2_{j} f $);
			\item $ \mu(\cdot) \in L^{\infty}((0,c), \mathbb{R}_+) $, $ k(\cdot) \in C_{b}([0,c), \mathbb{R}_+) $ (with $ \inf_{a>0} k(a) > 0 $); and
			\item $ \beta(\cdot, \cdot) \in C((0,c), L^{\infty}(\mathbb{R}^n)) \cap L^{p'}((0,c), L^{\infty}(\mathbb{R}^n)) $ ($ 1/{p'} + 1/p = 1 $).
		\end{enumerate}
		Note that the evolution family generated by $ \{A(a)\} $ is $ U(t,s) = e^{-\int_{s}^{t}\mu(r)~\mathrm{d}r} G(\int_{s}^{t} k(r) ~\mathrm{d}r) $ ($ t \geq s $), where $ G(\cdot) $ is the Gaussian semigroup generated by $ \Delta_s $, i.e.,
		\[
		(G(t)f)(x) : = (4\pi t)^{-n/2} \int_{\mathbb{R}^n} f(x - y) e^{-|y|^2/(4t)} ~\mathrm{d} y, ~ f \in L^q(\mathbb{R}^n).
		\]
		Clearly, $ U(t, s) $ is norm continuous on $ 0\leq s < t $ but $ U(a, 0) $ is not compact. The condition on $ \beta(\cdot, \cdot) $ implies $ C(\cdot) $ is norm continuous (but in general it is not uniformly norm continuous and also in this case $ L $ may be not compact). Therefore, \autoref{thm:age} (a) holds; however the result given in \cite[Section 5]{BHM05} cannot be applied (even for $ p = 1 $). Particularly, we have
		\[
		\omega_{\mathrm{crit}}(T_{(\mathcal{A}+\mathcal{L})_0}) = \omega_{\mathrm{crit}}(T_{0}) \leq \omega(U) \leq -|\mu(\cdot)|_{L^{\infty}}.
		\]

		Next, we show in this case $ \mathcal{A} $ usually is not a Hille--Yosida operator if $ p > 1 $ and $ c = \infty $. For simplicity, let $ n = 1 $, $ \mu(\cdot) \equiv 1 $ and $ k(\cdot) \equiv 1 $. Take $ f_1(x) = e^{-|x|^2} $. Then, $ f_1 \in L^p(\mathbb{R}) $ and for $ x > 0 $,
		\begin{align*}
		(U(a,0)f_1)(x) = e^{-a} (4\pi a)^{-1/2} \int_{-\infty}^{\infty} e^{-|x-y|^2} e^{-|y|^2/(4a)} ~\mathrm{d} y \geq c_0 e^{-a} e^{-x^2} (1+4a)^{-1/2},
		\end{align*}
		where $ c_0 > 0 $ is a constant. So
		\begin{align*}
		|(U(a,0)f_1)(\cdot)|_{L^q} = \left(\int_{-\infty}^{\infty} ((U(a,0)f_1)(x))^p ~\mathrm{d}x \right)^{1/q}\geq c_1 e^{-a} (1+4a)^{-1/2} : = h(a),
		\end{align*}
		where $ c_1 > 0 $ is a constant, and in particular, $ \limsup_{\lambda \to \infty} \lambda^p \int_{0}^{\infty} e^{-\lambda p a} h^p(a) da = \infty $ (in fact $ \lim_{\lambda \to \infty} \lambda \int_{0}^{\infty} e^{-(\lambda + 1) p a} (1+4a)^{-p/2} da = 1/p $). Hence, by \autoref{rmk:noHY}, in this case $ \mathcal{A} $ is not a Hille--Yosida operator.

		\item (See \cite{Rha98} and \cite[Section 5]{Thi98} in $ L^1 $ case.) Let $ E = L^q(\Omega) $ ($ 1 \leq q < \infty $) where $ \Omega $ is a bounded domain of $ \mathbb{R}^n $ with smooth boundary $ \partial \Omega $. Let $ A(\cdot) $ (with suitable boundary condition e.g., Dirichlet, Neumann, Robin, etc) and $ C(\cdot) $ as be in \eqref{exbb} with $ \mathbb{R}^n $ in (i) (ii) (iii) replaced by $ \Omega $. In this context, $ \{A(a)\} $ also generates exponentially bounded linear evolution family $ \{U(a, s)\}_{0 \leq s \leq a < c} $ (see, e.g., \cite{Paz83}) and the embedding theorem gives that $ U(t, s) $ is compact on $ 0\leq s < t < c $ and particularly \autoref{thm:age} (b) holds. More generally, $ \{A(a)\} $ can be (uniformly strongly) elliptic differential operators with suitable boundary condition (see, e.g., \cite[Section 7.6]{Paz83}). Due to the positive setting in \cite[Section 5]{Thi98}, $ \{C(a)\} $ can be unbounded which \autoref{thm:age} cannot cover; the proof given here is different from \cite{Thi98, Rha98}.

		\item (See \cite[Section 6]{BHM05} in $ L^1 $ case.) Consider the following model:
		\begin{equation*}
		\begin{cases}
		(\partial_t + \partial_a)u(t, a, s) = -\mu(a)u(t, a, s) + k(a) q(s) u(t, a, s), ~t > 0, a \in (0, c), s \in \Omega, \\
		u(t, 0, s) = \int_{0}^{c} \beta(a,s) u(t, a, s) ~\mathrm{d} a, ~ t > 0, s \in \Omega, \\
		u(0, a, s) = u_0(a, s), ~ a \in (0, c), s \in \Omega, ~ u_0 \in L^p((0,c), L^q(\Omega)).
		\end{cases}
		\end{equation*}
		Let $ E = L^q(\Omega) $ ($ 1 \leq q < \infty $) be endowed with a $ \sigma $-finite (Borel) measure where $ \Omega $ is a domain of $ \mathbb{R}^n $. Take $ A(a) = -\mu (a) I + k(a) q(\cdot) $ and $ C(a) = \beta (a, \cdot) $ where $ \mu(\cdot), k(\cdot), \beta(\cdot, \cdot) $ are the same as \eqref{exbb} with $ \mathbb{R}^n $ in (i), (ii), (iii) replaced by $ \Omega $. The measurable function $ q(\cdot): \Omega \to \mathbb{C} $ satisfies the following:
		\begin{enumerate}[({a}1)]
			\item $ \sup\{ \mathrm{Re}\lambda : \lambda \in q(\Omega) \} < \infty $ and
			\item for any $ d \in \mathbb{R} $, $ \overline{q(\Omega)} \cap \{ \lambda \in \mathbb{C}:\mathrm{Re}\lambda \geq d \} $ is bounded.
		\end{enumerate}
		Under this circumstance, the evolution family generated by $ \{A(a)\} $ is $ U(t,s) = e^{-\int_{s}^{t}\mu(r) - k(r)q(\cdot)~\mathrm{d}r} $ ($ t \geq s $). Note that by (a2), $ U(t,s) $ is norm continuous on $ 0\leq s < t < c $ (see, e.g., \cite[p. 121]{EN00}) but in general is not compact. So \autoref{thm:age} (a) holds.

		\item (See \cite[Section 1.4]{Web08} in $ L^1 $ case.)
		Consider the following model:
		\begin{equation*}
		\begin{cases}
		(\partial_t + \partial_a + k(a)\partial_s)u(t, a, s) = -\mu(a)u(t, a, s), ~t > 0, a \in (0, c), s \in (0,s_2), \\
		u(t, 0, s) = \int_{0}^{c} \int_{0}^{s_2} \beta(a,\hat{s},s) u(t, a,\hat{s}) ~\mathrm{d}\hat{s} ~\mathrm{d} a, ~ t > 0, s \in (0,s_2), \\
		u(0, a, s) = u_0(a, s), ~ a \in (0, c), s \in (0,s_2), ~ u_0 \in L^p((0,c), L^q(0, s_2)), \\
		u(t,a,0) = 0, ~t > 0, a \in (0, c).
		\end{cases}
		\end{equation*}
		Let $ E = L^q(0, s_2) $ ($ 1 \leq q < \infty $) where $ 0 < s_2 < \infty $. Let $ p', q' $ satisfy $ 1/{p'} + 1/p = 1 $ and $ 1/{q'} + 1/q = 1 $.
		The measurable function $ \beta(\cdot, \cdot, \cdot) $ satisfies
		\[\tag{b1}
		\beta(\cdot, \cdot, \cdot) \in C( (0,c), L^{(q, q')}((0,s_2)^2) ) \cap L^{P}( (0,c) \times (0,s_2)^2 ) , ~P = (q, q', p'),
		\]
		where $ L^{P}( (0,c) \times (0,s_2)^2 ) $ (similarly for $ L^{(q, q')}((0,s_2)^2) $) is the mixed-norm Lebesgue space (see, e.g., \cite{BP61}) defined by
		\[
		\begin{split}
		L^{P}( (0,c) \times (0,s_2)^2 )
		:= \{ f:(0,c) \times(0,s_2)^2 \to \mathbb{R} ~\text{is measurable}~:\\ \left(\int_{0}^{c} \left( \int_{0}^{s_2} \left(\int_{0}^{s_2} |f(a,\hat{s}, s)|^{q} ~\mathrm{d} s \right)^{q'/q} \mathrm{d} \hat{s}  \right)^{p'/{q'}}\mathrm{d}a \right)^{1/{p'}} < \infty \}.
		\end{split}
		\]
		In addition, if one of $ p,q $ equals $ 1 $, then we assume
		\[\tag{b2}
		\begin{cases}
		\lim_{h \to 0^+}\sup\limits_{(a,\hat{s}) \in (0,c) \times(0,s_2) } \int_{0}^{s_2 - h} |\beta(a, \hat{s}, s+h) - \beta(a, \hat{s}, s)| ~\mathrm{d}s = 0, ~ \text{if}~ p = q = 1,  \\
		\lim_{h \to 0^+} \int_{0}^{s_2} \left(\sup\limits_{\hat{s} \in (0,s_2) } \int_{0}^{s_2 - h} |\beta(a, \hat{s}, s+h) - \beta(a, \hat{s}, s)| ~\mathrm{d}s \right)^{p'} ~\mathrm{d}a = 0, ~\text{if}~ p > 1, q = 1,\\
		\lim_{h \to 0^+}\sup\limits_{a \in (0,c)} \int_{0}^{s_2} \left(\int_{0}^{s_2 - h} |\beta(a, \hat{s}, s+h) - \beta(a, \hat{s}, s)|^{q} ~\mathrm{d}s\right)^{q'/q} ~\mathrm{d} \hat{s} = 0, ~ \text{if}~ p = 1, q > 1.
		\end{cases}
		\]
		(If $ c < \infty $, then a typical case such that $ \beta(\cdot, \cdot, \cdot) $ satisfies (b1) (b2) for all $ 1 \leq p, q < \infty $ is $ \beta(\cdot, \cdot, \cdot) \in C([0,c]\times[0,s_2]^2) $.)

		Take $ A(a) = - \mu(a) I - k(a) \frac{\mathrm{d}}{\mathrm{d}s} $ where $ \mu(\cdot), k(\cdot) $ are the same as \eqref{exbb} with $ \mathbb{R}^n $ in (i), (ii), (iii) replaced by $ (0, s_2) $; $ \frac{\mathrm{d}}{\mathrm{d}s} $ is the first-order differential operator on $ L^q(0,s_2) $, i.e.,
		\[
		\frac{\mathrm{d}}{\mathrm{d}s}: f \mapsto \frac{\mathrm{d}f}{\mathrm{d}s}, ~D(\frac{\mathrm{d}}{\mathrm{d}s}) = \{f \in W^{1,q}(0,s_2): f(0) = 0\}.
		\]
		Let $ C(a) = \beta(a, \cdot, \cdot) $ be defined by
		\[
		(C(a) \phi)(s) = \int_{0}^{s_2} \beta(a, \hat{s}, s) \phi(\hat{s}) ~\mathrm{d} \hat{s}, ~\phi \in L^q(0, s_2);
		\]

		By Minkowski integral inequality, we see $ C(a) $ indeed is a bounded linear operator on $ L^q(0, s_2) $. Clearly, $ \{A(a)\} $ generates the exponentially bounded linear evolution family $ \{U(a, s)\}_{0 \leq s \leq a < c} $ defined by $ U(t,s) = e^{-\int_{s}^{t}\mu(r)~\mathrm{d}r} T_{r}(\int_{s}^{t} k(a) ~\mathrm{d}a) $, where $ T_r(t) $ is the right translation on $ L^q(0,s_2) $, i.e.,
		\[
		(T_r(t) \phi)(s) : = \begin{cases}
		\phi(s - t), & s - t \geq 0,\\
		0,& s - t < 0.
		\end{cases}
		\]
		The condition on $ \beta(\cdot, \cdot, \cdot) $ implies that $ L $ is compact, and so in this case \autoref{thm:age} (c) holds and
		especially, we have
		\[
		\omega_{\mathrm{ess}}(T_{(\mathcal{A}+\mathcal{L})_0}) = \omega_{\mathrm{ess}}(T_{0}) \leq \omega(U) \leq -|\mu(\cdot)|_{L^{\infty}}.
		\]
		Notice also that $ U(a, 0) $ is not compact for all $ a > 0 $.
		\begin{proof}
			To show $ L $ is compact, one can use the classical Kolmogorov theorem on the characterization of the relatively compact subsets in $ L^q(0,s_2) $ which is standard; the details are as follows. We need to show $ \{ L\varphi: \varphi \in L^{p}((0,c), L^q(0,s_2)) ~\text{such that}~ |\varphi| \leq 1 \} $ is relatively compact in $ L^q(0,s_2) $, or equivalently,
			\[
			\lim_{h\to 0^+} \int_{0}^{s_2-h} |(L\varphi)(s+h) - (L\varphi)(s)|^{q} ~\mathrm{d} s = 0,~ ~\text{uniformly for}~ |\varphi| \leq 1.
			\]
			For $ \varphi \in L^{p}((0,c), L^q(0,s_2)) = L^{(q,p)}((0,c)\times(0,s_2)) $ such that $ |\varphi| \leq 1 $, i.e.,
			\[
			\left(\int_{0}^{c}\left(\int_{0}^{s_2} |\varphi(a,\hat{s})|^{q}~\mathrm{d} \hat{s}\right)^{p/q} ~\mathrm{d}a \right)^{1/p} \leq 1,
			\]
			using Minkowski integral inequality and H\"older inequality, we get
			\begin{align*}
			& \int_{0}^{s_2-h} |(L\varphi)(s+h) - (L\varphi)(s)|^{q} ~\mathrm{d} s \\
			= & \int_{0}^{s_2-h} \left|\int_{0}^{c}\int_{0}^{s_2} \{\beta(a, \hat{s}, s + h)  - \beta(a, \hat{s}, s)\} \varphi(a,\hat{s}) ~\mathrm{d} \hat{s} ~\mathrm{d}a\right|^{q} ~\mathrm{d} s \\
			\leq & \int_{0}^{c}\int_{0}^{s_2} \left(\int_{0}^{s_2-h} |\beta(a, \hat{s}, s + h)  - \beta(a, \hat{s}, s)|^{q} ~\mathrm{d} s\right)^{1/q} |\varphi(a,\hat{s})|~\mathrm{d} \hat{s} ~\mathrm{d}a \\
			\leq & \int_{0}^{c} \left(\int_{0}^{s_2} \left(\int_{0}^{s_2-h} |\beta(a, \hat{s}, s + h)  - \beta(a, \hat{s}, s)|^{q} ~\mathrm{d} s\right)^{q'/q} ~\mathrm{d} \hat{s}\right)^{1/q'} \cdot \left(\int_{0}^{s_2} |\varphi(a,\hat{s})|^{q}~\mathrm{d} \hat{s}\right)^{1/q} ~\mathrm{d}a \\
			\leq & \left(\int_{0}^{c} \left(\int_{0}^{s_2} \left(\int_{0}^{s_2-h} |\beta(a, \hat{s}, s + h)  - \beta(a, \hat{s}, s)|^{q} ~\mathrm{d} s\right)^{q'/q} ~\mathrm{d} \hat{s}\right)^{p'/q'} ~\mathrm{d}a\right)^{1/p'} : = \rho(h).
			\end{align*}
			Now we have $ \rho(h) \to 0 $ as $ h \to 0^+ $. Indeed, if $ p \neq 1 $ and $ q \neq 1 $, then by a standard argument we can obtain this (see, e.g., \cite[Section 10 Theorem 1]{BP61}); if one of $ p,q $ equals $ 1 $, then this is condition (b2). The proof is complete.
		\end{proof}
	\end{asparaenum}
\end{exa}

\begin{rmk}
	In \autoref{exa:model} (b), (c), (d), if $ p = 1 $ and $ \beta(\cdot, \cdot) \in BUC([0,c], L^{\infty}(\Omega)) $ ($ \Omega = \mathbb{R}^n $ or a domain of $ \mathbb{R}^n $) as \cite{Rha98, BHM05}, then \autoref{thm:C} can be applied directly.
	A more general model than \autoref{exa:model} (e) has been considered in \cite[Section 6]{Thi98}.
\end{rmk}

\section{Comments}\label{comments}
\subsection{Unbounded perturbation}
The unbounded perturbation theorem for MR operators (quasi Hille--Yosida operators) is few; see \cite[Section 2]{TV09} for certain unbounded perturbations of Hille--Yosida operators. It's still a difficult task. For the case of the generators of integrated semigroups and almost sectorial operators, we refer the readers to see \cite{ABHN11, KW03, Thi08} and the references therein. Here, we give an unbounded perturbation theorem for MR operators, which is essentially due to Arendt et al. \cite{ABHN11}.
\begin{thm}[{\cite[Theorem 3.5.7]{ABHN11}}]
	Assume $A$ is an operator on a Banach space $X$ such that $(\omega,\infty)\subset \rho(A)$, $\limsup\limits_{\lambda > \omega, \lambda \rightarrow \infty} \|R(\lambda, A)\| = 0$, and $B \in \mathcal{L}(D(A), D(A))$ ($D(A)$ is endowed with the graph norm). Then, there is $C \in \mathcal{L}(X)$ such that $A+B$ is similar to $A+C$.
\end{thm}
\begin{proof}
	The proof is essentially the same as \cite[Theorem 3.5.7]{ABHN11}. In fact, the condition $\sup\limits_{\lambda > \omega} \|\lambda R(\lambda, A)\| < \infty$ in \cite[Theorem 3.5.7]{ABHN11} can be weakened by $\limsup\limits_{\lambda > \omega,\lambda \rightarrow \infty} \|R(\lambda, A)\| = 0$, because we only need $I - BR(\lambda, A) $ is invertible for sufficiently large $\lambda$. We repeat the proof as follows for the sake of readers.

	Take $ \lambda_0 > \omega $. Since $ (\lambda_0 - A) B R(\lambda_0, A) \in \mathcal{L}(X, X) $, we have $ I - (\lambda_0 - A) B R(\lambda, A) R(\lambda_0, A) = I - (\lambda_0 - A) B R(\lambda_0, A)R(\lambda, A) $ is invertible for large $\lambda$ as $ \limsup\limits_{\lambda > \omega, \lambda \rightarrow \infty} \|R(\lambda, A)\| = 0 $. Now $ I - BR(\lambda, A) = I - R(\lambda_0, A) (\lambda_0 - A) B R(\lambda, A) $ is invertible (here, we use the fact: for $ U, V \in \mathcal{L}(X, X) $, $ I - UV $ is invertible if and only if $ I - VU $ is invertible).

	Take $C := (\lambda - A)BR(\lambda, A)$ and $ U := I - BR(\lambda, A) $ for sufficiently large $\lambda$. Then, as shown above $ U $ is invertible (and $ U D(A) = D(A) $). It is easy to verify that $ U (A + C) U^{-1} = A + B $ (see also the proof \cite[Theorem 3.5.7]{ABHN11}). The proof is complete.
\end{proof}

Since the condition $\limsup\limits_{\lambda > \omega, \lambda \rightarrow \infty} \|R(\lambda, A)\| = 0$ can be satisfied by MR operators (see \autoref{lem:MR} (b)), we obtain the following result.
\begin{cor}
	If $A$ is an MR operator (resp. $p$-quasi Hille--Yosida operator) and $B \in \mathcal{L}(D(A), D(A))$, so is $A+B$.
\end{cor}

It's possible for us to extend the results in sections 3-5 for this unbounded perturbation. We take $L=(\lambda - A)BR(\lambda, A)$. Here are some examples.
\begin{cor}
	Let $A$ is an MR operator, and $B \in \mathcal{L}(D(A), D(A))$.
	\begin{enumerate}[(a)]
		\item If there exists $\lambda_0 \in \rho(A)$ such that $(\lambda_0 - A)BR(\lambda_0, A_0) T_{A_0}$ is norm continuous (resp. norm continuous and compact) on $(0,\infty)$, then $\omega_{\mathrm{crit}}(T_{(A+B)_0}) = \omega_{\mathrm{crit}}(T_{A_0})$ (resp. $\omega_{\mathrm{ess}}(T_{(A+B)_0}) = \omega_{\mathrm{ess}}(T_{A_0})$).
		\item If $A$ is also a quasi Hille--Yosida operator and there is $\lambda_0 \in \rho(A)$ such that $(\lambda_0 - A)BR(\lambda_0, A_0) T_{A_0}$ is compact on $(0,\infty)$, then $\omega_{\mathrm{ess}}(T_{(A+B)_0}) = \omega_{\mathrm{ess}}(T_{A_0})$.
	\end{enumerate}
\end{cor}
\begin{proof}
	For (a) (b), it suffices to show if $(\lambda_0 - A)BR(\lambda_0, A_0)T_{A_0}$ is norm continuous (resp. compact) on $(0,\infty)$, so is $(\lambda - A)BR(\lambda, A_0)  T_{A_0}$ for all $\lambda \in \rho(A)$.
	Note that
	\[
	R(\lambda, A_0) T_{A_0} = T_{A_0} R(\lambda, A_0),
	\]
	and
	\[
	BR(\lambda_0, A) T_{A_0} = R(\lambda_0, A)(\lambda_0 - A)BR(\lambda_0, A) T_{A_0}
	\]
	has the same regularity as $(\lambda_0 - A)BR(\lambda_0, A) T_{A_0}$.
	By
	\begin{align*}
		(\lambda-A)BT_{A_0}R(\lambda, A_0) & =  (\lambda-\lambda_0+\lambda_0-A)BT_{A_0} [R(\lambda_0, A_0) + (\lambda_0 - \lambda)R(\lambda_0, A_0)R(\lambda, A_0)] \\
		& =  [ (\lambda-\lambda_0)BT_{A_0}R(\lambda_0, A_0) + (\lambda_0-A)BT_{A_0}R(\lambda_0, A_0)] [I+(\lambda_0-\lambda)R(\lambda,A)],
	\end{align*}
	we obtain the results.
\end{proof}

\subsection{Cauchy problems}
Our consideration is relevant to the following semilinear Cauchy problem:
\begin{equation}\label{cauchy}
\begin{cases}
\dot{u}(t) = Au(t) + f(u(t)), \\
u(0) = x,
\end{cases}
\end{equation}
where $A:~D(A) \subset X \rightarrow X$ is an MR operator and $f:\overline{D(A)} \rightarrow X$ is $ C^1 $ and globally Lipschitz (for simplicity in order to avoid blowup). Many differential equations such as age-structured population models, parabolic differential equations, delay equations, and Cauchy problems with boundary conditions can be reformulated as Cauchy problems. However, in many cases, the operator $A$ is not densely defined or even not a Hille--Yosida operator, see, e.g., \cite{DPS87, PS02, DMP10, MR18}.

Assume zero is an equilibrium of \eqref{cauchy} (i.e., $f(0) = 0$). Let $L := Df(0) \in \mathcal{L}(\overline{D(A)}, X)$.
Consider the linearized equation of \eqref{cauchy}, i.e.,
\begin{equation}\label{lincauchy}
\begin{cases}
\dot{u}(t) = Au(t) + Lu(t), \\
u(0) = x.
\end{cases}
\end{equation}
It was shown in \cite{MR09} that solutions of \eqref{lincauchy} can reflect the properties of solutions of \eqref{cauchy} in the neighborhood of $0$. In general, the properties of $A$ (or $T_{A_0}$, $S_A$) would be known well. What we need is the properties of $T_{(A+L)_0}$. Our results can be applied to this situation.
For example, if $T_{(A+L)_0}$ is exponentially stable (i.e., $\omega (T_{(A+L)_0}) < 0$ (see \eqref{equ:bound})), then the zero solution of \eqref{cauchy} is locally stable \cite[Proposition 7.1]{MR09}. If $\omega_{\mathrm{crit}}(T_{A_0}) < 0$ and $\omega_{\mathrm{crit}}(T_{(A+L)_0}) \leq \omega_{\mathrm{crit}}(T_{A_0})$, then $\omega_{\mathrm{crit}}(T_{A_0}) < 0$. Therefore, all we need to consider is: $s(A+L) < 0$ (see \eqref{equ:bound})? (see \autoref{thm:crit} (c), and note that $s(A+L) = s((A+L)_0)$.) That is, the spectrum of $A+L$ can reflect the stability of the zero solution.
\begin{thm}\label{thm:final}
	\begin{enumerate}[(a)]
		\item Assume the condition of \autoref{thm:main1} (a) (or \autoref{thm:main2} (a), or \autoref{cor:lt} (a)) is satisfied and $\omega_{\mathrm{crit}}(T_{A_0}) < 0$, then the zero solution of \eqref{cauchy} is locally exponentially stable, provided $s(A+L) < 0$.
		\item Assume the condition of \autoref{thm:main1} (b) (or \autoref{thm:main2} (b), or \autoref{cor:lt} (b), or \autoref{cor:lmr}) is satisfied and $\omega_{\mathrm{ess}}(T_{A_0}) < 0$. Then, the zero solution of \eqref{cauchy} is locally exponentially stable if $s(A+L) < 0$ and unstable if $s(A+L) > 0$.
	\end{enumerate}
\end{thm}
\begin{proof}
	(a) This follows from \autoref{thm:crit} (a) and \cite[Proposition 7.1]{MR09}.

	(b) This follows from \autoref{thm:ess} (a) and \cite[Proposition 7.1, Proposition 7.4]{MR09}.
\end{proof}

The existence of the Hopf bifurcation and the center manifold (of an equilibrium) for Cauchy problems needs the condition $\omega_{\mathrm{ess}}(T_{(A+L)_0}) < 0$, see, e.g., \cite{MR09a}; this condition was replaced by the \emph{exponential dichotomy condition} instead in our paper \cite{Che18c} to give the invariant manifold theory around more general manifolds (in the sense of Hirsch, Pugh and Shub, and Fenichel).
A more concrete application of the results in \autoref{criess} and \autoref{regper} to a class of delay equations with non-dense domains, see \cite{Che18g}, which is very similar as \cite{BMR02}.




\end{document}